\documentclass[preprint,12pt]{elsarticle}
\usepackage{amsmath}
\usepackage{amsfonts}
\usepackage{amssymb}
\DeclareMathOperator{\diff}{\mathrm{d} \! }
\newtheorem{thm}{Theorem}

\newproof{proof}{Proof}
\newdefinition{remark}{Remark}

\journal{Journal of Computational Physics}

\begin{document}


\title{Embedded WENO: a design strategy to improve existing WENO schemes}

\author[tue]{Bart S. van Lith\fnref{fn1}}
\ead{b.s.v.lith@tue.nl}
\author[tue]{Jan H.M. ten Thije Boonkkamp}
\author[tue,phl]{Wilbert L. IJzerman}

\fntext[fn1]{Corresponding author}

\address[tue]{Department of Mathematics and Computer Science, Eindhoven University of Technology - P. O. Box 513, NL-5600 MB Eindhoven, The Netherlands.}
\address[phl]{Philips Lighting - High Tech Campus 44, 5656 AE, Eindhoven, The Netherlands.}

\begin{abstract}
Embedded WENO methods utilize \textit{all} adjacent smooth substencils to construct a desirable interpolation. Conventional WENO schemes under-use this possibility close to large gradients or discontinuities. We develop a general approach for constructing embedded versions of existing WENO schemes. Embedded methods based on the WENO schemes of Jiang and Shu \cite{jiang_shu} and on the WENO-Z scheme of Borges et al. \cite{borges} are explicitly constructed. Several possible choices are presented that result in either better spectral properties or a higher order of convergence for sufficiently smooth solutions. However, these improvements carry over to discontinuous solutions. The embedded methods are demonstrated to be indeed improvements over their standard counterparts by several numerical examples. All the embedded methods presented have no added computational effort compared to their standard counterparts.
\end{abstract}


\maketitle

\section{Introduction}

In a seminal paper in 1987, Harten and Osher introduced the essentially non-oscillatory (ENO) reconstruction technique \cite{harten_osher}. The basic idea of ENO is to construct several different candidate polynomial interpolations and to choose the smoothest approximation to work with. The choice is facilitated by means of smoothness indicators, which become larger as the interpolation varies more rapidly.\par
Building on the ENO scheme, Liu, Osher and Chan introduced the weighted essentially non-oscillatory (WENO) reconstruction technique in 1994 \cite{liu_osher}. The WENO technique comes from the realization that the three approximations of ENO can be combined to construct a higher-order approximation. Instead of the logical statements inherent in the ENO scheme, the WENO scheme weighs every lower-order approximation according to its smoothness indicator. Thus, in smooth regions, WENO gives a better approximation, while reducing to ENO near discontinuities.\par
WENO schemes are ubiquitous in science and engineering, with applications in fluid dynamics, astrophysics, or any other application involving convection-dominated dynamics \cite{shu_FD,qiu}. The technique is mainly applied in the context of hyperbolic and convection-dominated parabolic PDEs. However, since it is a highly advanced interpolation technique, it also has applications in fields that do not use it as part of a PDE solver, such as computer vision and image processing \cite{amat,siddiqi}.\par
The standard WENO scheme as it is most commonly used today was devised by Jiang and Shu \cite{jiang_shu}, and is sometimes referred to as the WENO-JS scheme. Recently, several variants of the WENO scheme have appeared that improve the order of accuracy near points where the first derivative vanishes. Examples include the WENO-M \cite{henrick,feng}, WENO-Z \cite{borges,castro,arandiga} and WENO-NS \cite{ha} schemes. For a comparison of the performance of these schemes, see Zhao et al. \cite{zhao}. Other efforts have focused on creating energy-stable WENO schemes such as those constructed by Yamaleev et al. \cite{yamaleev, yamaleev2}, or decreasing numerical dissipation by using central discretisations such as considered by Hu et al. \cite{hu}.\par
The most common implementations of WENO schemes use a five-point stencil, which can be subdivided into three three-point stencils. WENO schemes switch seamlessly between the third- and fifth-order reconstructions that are possible on the five-point stencil. The idea is straightforward: when all three smoothness indicators are roughly equal, a WENO scheme switches to the fifth-order mode. When one or more smoothness indicators are large, a WENO scheme switches to the third-order mode.\par
In this formulation, it seems obvious that information is discarded when only one out of three smoothness indicators is large. When this happens, the two smooth approximations could still be used to obtain better accuracy. The current WENO methods do not allow for control over the numerical solution in this situation. However, one very recent scheme which does feature this type of functionality is the targeted ENO scheme of Fu et al. \cite{fu}. Their approach is completely novel and uses a combination of ideas from ENO and WENO schemes. In our work, we propose a design strategy that aims to adapt existing WENO schemes such that they utilise the maximum number of grid points that form a smooth substencil. Moreover, we shall explicitly construct variants of two existing WENO schemes that exhibit this property.\par
Apart from the order of convergence, one can also analyse a WENO scheme in terms of its spectral properties \cite{jia}. WENO schemes switch non-linearly between linear modes of operation and as such, it is possible to investigate the spectral properties by analysing the underlying linear schemes \cite{martin}. We will also show that our method allows for tuning of spectral properties such as dispersion and dissipation.\par
This paper is arranged in the following way: in Section \ref{sec:WENO_recap} we give a short recap of WENO methods, in Section \ref{sec:embedded_WENO} we introduce the embedding method, which is implemented in Section~\ref{sec:implementation}. In Section \ref{sec:spectral} we look at the spectral properties of the embedded schemes and in Section \ref{sec:numerical_experiments} we show results of several numerical experiments. Finally, we present our conclusions and outlook in Section \ref{sec:conclusion}.

\section{The classical WENO scheme}\label{sec:WENO_recap}
The WENO method is an advanced interpolation technique that aims to suppress spurious oscillations. It is commonly used as part of a high-resolution scheme for hyperbolic conservation laws, e.g.,
\begin{equation}
\frac{\partial u}{\partial t} + \frac{\partial}{\partial x} f(u)  = 0,
\end{equation}
where $f$ is the flux function. To obtain numerical solutions, we introduce a grid, $\{ x_j \}_{j=1}^N$, with grid size $\Delta x$. With each point $x_j$, we associate a cell centred on $x_j$ of width $\Delta x$, i.e., the interval $(x_{j-\frac{1}{2}},x_{j+\frac{1}{2}})$. Taking the average of the conservation law over cell $j$, we find
\begin{equation}\label{eq:finite_volume}
\frac{\diff u_j}{\diff t} + \frac{1}{\Delta x} \left( f(u(x_{j+\frac{1}{2}},t)) - f(u(x_{j-\frac{1}{2}},t)) \right) = 0,
\end{equation}
where $u_j$ is the average value of $u$ over cell $j$. Note that this ODE for the average value $u_j$ is exact as long as we know the exact value of $u$ on the cell boundaries. We shall, in the following, suppress the explicit time dependence of $u$, as we interpolate $u$ in space for fixed time. In a numerical scheme, we introduce a numerical flux function to represent the fluxes on the cell edges. Regardless of the choice of numerical flux, we require the value of $u$ at the cell interfaces $x_{j \pm \frac{1}{2}}$, i.e. $u(x_{j \pm \frac{1}{2}})$. However, if $u$ is discontinuous and we would naively use polynomial interpolation, we inadvertently introduce spurious oscillations. A (W)ENO scheme is a more advanced interpolation technique that is designed to suppress these oscillations.\par
The classical WENO scheme, or WENO-JS, can be constructed by considering a five-point stencil around $x_j$, i.e., $S = \{ x_{j-2}, x_{j-1},x_j , x_{j+1}, x_{j+2} \}$. The large stencil can be divided into three smaller substencils, viz., $S_0 = \{ x_{j-2}, x_{j-1},x_j\}$, $S_1 = \{x_{j-1},x_j , x_{j+1} \}$ and $S_2 = \{x_j , x_{j+1}, x_{j+2} \}$; see Figure~\ref{fig:WENO_stencils}.

\begin{figure}
\centering
\includegraphics[width = .5\textwidth]{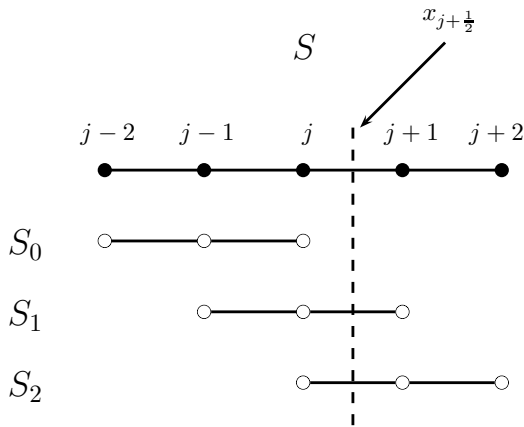}
\caption{The five-point stencil $S$, with substencils $S_0$, $S_1$ and $S_2$. Note that the stencil is asymmetric around the interpolation point.}
\label{fig:WENO_stencils}
\end{figure}

On each of these substencils, $S_k$ with $k=0,1,2$, we can approximate $u(x_{j+\frac{1}{2}})$ by constructing a second-degree polynomial $p_k$ that has the same cell averages as $u$, i.e.,
\begin{equation}
\frac{1}{\Delta x} \int\limits_{x_{j-\frac{1}{2}}}^{x_{j+\frac{1}{2}}} p_k(x) \diff x = u_j.
\end{equation}
Introducing the auxiliary vector $\mathbf{v} = (u_{j-2},u_{j-1},u_j,u_{j+1},u_{j+2})^T $ representing the averages on the large stencil $S$, the three lower-order approximations of $u(x_{j+\frac{1}{2}})$ can be represented as
\begin{equation}\label{eq:third_order_approximations}
\mathbf{u}_{j+\frac{1}{2}} = C \mathbf{v},
\end{equation}
with $C$ a $3 \times 5$ matrix. It is straightforward to show that one can obtain a fifth-order approximation by taking a linear combination of the third-order approximations in \eqref{eq:third_order_approximations}. The fifth-order upwind approximation is therefore given by
\begin{equation}\label{eq:fifth_linear_weights}
u^{(UW5)}_{j+\frac{1}{2}}  = \boldsymbol{\gamma}^T C \mathbf{v},
\end{equation}
with $\boldsymbol{\gamma}$ being a column vector. The matrix $C$ and the linear, also called optimal, weights $\boldsymbol{\gamma}$, can be represented in a tableau inspired by Butcher Tableaux \cite{butcher},
\begin{equation}
\begin{array}{c|c}
C & \boldsymbol{\gamma} \\ \hline
\end{array}.
\end{equation}
Organised this way, it contains all the coefficients involved in a WENO scheme, thus giving a concise overview of the underlying linear method. The tableau for the five-point WENO scheme looks as follows \cite{jiang_shu},
\begin{equation}\label{eq:WENO5_tableau}
\begin{array}{ccccc|c}
\frac{2}{6} & -\frac{7}{6} & \frac{11}{6} & & & \frac{1}{10} \\ \rule{0pt}{2.6ex}
& -\frac{1}{6} & \frac{5}{6} & \frac{2}{6} & & \frac{6}{10}  \\ \rule{0pt}{2.6ex}
& & \frac{2}{6} & \frac{5}{6} & -\frac{1}{6} & \frac{3}{10} \\
\hline
\end{array} \begin{matrix} \vphantom{x} \\ \vphantom{y} \\ \vphantom{z}.\end{matrix}
\end{equation}
The previous discussion shows how a fifth-order linear approximation can be constructed from three third-order underlying approximations. However, whenever there is a discontinuity on the stencil, the fifth-order approximation incurs spurious oscillations and a third-order approximation might actually be better in some sense. Thus, a set of nonlinear weights are needed that take into account the smoothness of each third-order approximation.\par
This idea can be realized by introducing smoothness indicators $\beta_k$, $k=0,1,2$. There are several smoothness indicators available in the literature \cite{fan_shen,zhang_shu}, each one exhibiting some desirable property. A very popular set of indicators, however, was introduced by Jiang and Shu \cite{jiang_shu} and is given by
\begin{equation}
\beta_k :=  \int\limits_{x_{j-\frac{1}{2}}}^{x_{j+\frac{1}{2}}} \left(  p_k^{\prime \prime} (x) \right)^2 \Delta x^3 + \left( p_k^{\prime} (x) \right)^2 \Delta x \diff x.
\end{equation}
A lengthy but straightforward calculus exercise shows that
\begin{subequations}\label{eq:smoothness_definition}
\begin{align}
\beta_0 &= \frac{13}{12}(u_{j-2} -2 u_{j-1} + u_j)^2 + \frac{1}{4}(u_{j-2} -4 u_{j-1} +3 u_j)^2,\\
\beta_1 &= \frac{13}{12}(u_{j-1} -2 u_{j} + u_{j+1})^2 + \frac{1}{4}(u_{j-1} - u_{j+1})^2 ,\\
\beta_2 &= \frac{13}{12}(u_{j} -2 u_{j+1} + u_{j+2})^2 + \frac{1}{4}(3u_{j} - 4 u_{j+1} + u_{j+2})^2,
\end{align}
\end{subequations}
where one can recognise undivided finite differences. Provided that $u$ is sufficiently smooth, a Taylor expansion reveals that $\beta_k = \mathcal{O}(\Delta x^2)$, where the coefficients of the expansion contain various derivatives of $u$, either squared or multiplied with higher order derivatives, i.e.,
\begin{subequations}\label{eq:smoothness_taylor}
\begin{align}
\beta_0 &= (u^{\prime}_j)^2 \Delta x^2 + \big(\tfrac{13}{12} (u^{\prime \prime}_j)^2 - \tfrac{2}{3} u_j^\prime u_j^{\prime \prime \prime}  \big) \Delta x^4 - \big(\tfrac{13}{6} u_j^{\prime \prime} u_j^{\prime \prime \prime} - \tfrac{1}{2} u_j^\prime u_j^{\prime \prime \prime \prime}  \big) \Delta x^5 + \mathcal{O}(\Delta x^6),\\
\beta_1 &= (u_j^\prime)^2 \Delta x^2 + \big( \tfrac{13}{12} (u^{\prime \prime}_j)^2 + \tfrac{1}{3} u_j^\prime u_j^{\prime \prime \prime}  \big) \Delta x^4 + \mathcal{O}(\Delta x^6),\\
\beta_2 &= (u^{\prime}_j)^2 \Delta x^2 + \big(\tfrac{13}{12} (u^{\prime \prime}_j)^2 - \tfrac{2}{3} u_j^\prime u_j^{\prime \prime \prime}  \big) \Delta x^4 + \big(\tfrac{13}{6} u_j^{\prime \prime} u_j^{\prime \prime \prime} + \tfrac{1}{2} u_j^\prime u_j^{\prime \prime \prime \prime}  \big) \Delta x^5 + \mathcal{O}(\Delta x^6),
\end{align}
\end{subequations}
where $u_j^\prime$ is shorthand for $ \partial_x u(x_j)$, etc. Whereas an ENO scheme uses a logical statement to select the interpolation with the lowest smoothness indicator, a WENO scheme proposes to use a convex combination of the third-order interpolations, much like \eqref{eq:fifth_linear_weights}. To this end, the nonlinear weights $\omega_k$ are introduced, which are functions of the smoothness indicators. Collecting the nonlinear weights into a column vector $\boldsymbol{\omega}$, a WENO scheme uses a linear combination of the form
\begin{equation}\label{eq:convex_combination}
u_{j+\frac{1}{2}}^\mathrm{(WENO)}  = \boldsymbol{\omega}^T C \mathbf{v}.
\end{equation}
Consistency requires that the nonlinear weights $\omega_k$ ($k=0,1,2$) sum to unity. Hence, to construct nonlinear weights that satisfy the requirements discussed earlier, we first compute the unnormalized nonlinear JS weights, indicated with a superscript JS, as
\begin{equation}\label{eq:unnormalised_JS}
\tilde{\omega}_k^\mathrm{JS} = \frac{\gamma_k}{(\beta_k + \varepsilon)^p},
\end{equation}
with $\varepsilon>0$ a small number to avoid division by zero and $p>0$. Typical values are $\varepsilon = 10^{-6}$ and $p=2$. In any WENO scheme the unnormalized weights are subsequently normalised to obtain the nonlinear weights
\begin{equation}\label{eq:normalised_JS}
\omega_k = \frac{\tilde{\omega}_k}{\sum\limits_{l=0}^2\tilde{\omega}_l}.
\end{equation}
The WENO-JS scheme gives fifth-order accuracy whenever $u$ is smooth, i.e.,  $u^\prime_j = \mathcal{O}(1)$ and consequently $\beta_k = \mathcal{O}(\Delta x^2)$, or if $\varepsilon$ is sufficiently large compared to the second-order terms in the expansions \eqref{eq:smoothness_taylor}, otherwise only third-order is attained \cite{henrick}. At the same time, it gives third-order accuracy whenever a substencil contains a discontinuity, since then the corresponding smoothness indicator becomes large. By choosing instead to use only one of the smooth substencils, oscillations are suppressed.\par
A modern incarnation of the WENO scheme is given by the WENO-Z scheme of Borges et al. \cite{borges}, who showed that a sufficient condition for fifth-order accuracy is
\begin{equation}\label{eq:WENO_sufficient}
\omega_k = \gamma_k + \mathcal{O}(\Delta x^3).
\end{equation}
The WENO-JS scheme attains $\omega_k = \gamma_k + \mathcal{O}(\Delta x^2)$, however it does satisfy a more complicated condition ensuring fifth-order accuracy. Unfortunately, near critical points, the WENO-JS scheme only provides third-order accuracy as pointed out by Henrick et al. \cite{henrick}. WENO-Z was designed to satisfy \eqref{eq:WENO_sufficient} and thereby restore optimal convergence near critical points. The unnormalised weights, indicated with a superscript Z, are given by
\begin{equation}\label{eq:WENO-Z_weights}
\tilde{\omega}^Z_k = \gamma_k \left( 1 + \left(\frac{\tau}{\beta_k + \varepsilon} \right)^p \right),
\end{equation}
where $\tau = | \beta_2 - \beta_0|$ is called the global smoothness indicator. Using \eqref{eq:smoothness_taylor}, one can show that $\tau = \mathcal{O}(\Delta x^5)$ and so WENO-Z satisfies the sufficient condition \eqref{eq:WENO_sufficient} for any $p \geq 1$.\par
WENO schemes are commonly employed in a method of lines (MOL) approach, where one leaves time continuous while discretising space. This approach then turns a PDE into a large number of coupled ODEs, resulting in a system of equations
\begin{equation}
\frac{\diff \mathbf{u}}{\diff t} = L(\mathbf{u}),
\end{equation}
where $L$ is the result of the application of the WENO scheme. After the spatial discretisation, one discretises time by setting time levels $t^n = n \Delta t$, $n =0,1,\ldots$. The time integrators of choice are the strong stability preserving Runge-Kutta methods (SSPRK) \cite{gottlieb,gottlieb_review}. These are explicit Runge-Kutta methods that have a high order of accuracy and do not incur spurious oscillations due to time integration. Throughout this paper, we shall use the fairly standard SSPRK(3,3) method, one time step of this method is given by
\begin{subequations}\label{eq:TVD-RK}
\begin{align}
\mathbf{u}^{(1)} &= \mathbf{u}^n + \Delta t L(\mathbf{u}^n),\\
\mathbf{u}^{(2)} &= \tfrac{3}{4} \mathbf{u}^n + \tfrac{1}{4} \mathbf{u}^{(1)} + \tfrac{1}{4} \Delta t L (\mathbf{u}^{(1)}) ,\\
\mathbf{u}^{n+1} &= \tfrac{1}{3} \mathbf{u}^n + \tfrac{2}{3} \mathbf{u}^{(2)} + \tfrac{2}{3} \Delta t L(\mathbf{u}^{(2)}) ,
\end{align}
\end{subequations}
where $\mathbf{u}^{(1)}$ and $\mathbf{u}^{(2)}$ are the intermediate stages. This method exhibits the strong stability preserving property and provides a third-order accuracy in time. Moreover, Wang and Rong \cite{wang} have shown that this method is linearly stable when combined with a five-point WENO scheme.

\section{Embedded WENO}\label{sec:embedded_WENO}
We now pose the question of what happens when the solution on two adjacent substencils is smooth with no critical points and the third one contains a discontinuity. Specifically, either the solution is smooth on $S_0$ and $S_1$ and not smooth on $S_2$, or the solution is smooth on $S_1$ and $S_2$ and not on $S_0$. The answer is that the WENO-JS scheme provides third-order accuracy while suppressing oscillations. However, the scheme generates a linear combination of the two smooth substencils that is forced by the fifth-order mode, i.e., the user cannot choose the resulting weights. Being able to choose the resulting weights results in direct control over the truncation error and the numerical dissipation and dispersion.\par
As a shorthand whenever the solution is smooth on a substencil $S_k$, we call the substencil smooth. Let us examine the normalised JS weights, indicated with the superscript $\mathrm{JS}$, from the definition \eqref{eq:unnormalised_JS} - \eqref{eq:normalised_JS} we find that
\begin{equation}
\frac{\omega_k^\mathrm{JS}}{\omega_l^\mathrm{JS}} = \frac{\tilde{\omega}_k^\mathrm{JS}}{\tilde{\omega}_l^\mathrm{JS}} = \frac{\gamma_k}{\gamma_l} \left(\frac{\beta_l }{\beta_k } \right)^p,
\end{equation}
where we have assumed $\varepsilon$ is negligible compared to the smoothness indicators. Thus, the proportions of the nonlinear weights only depend on the local smoothness of $S_k$ and $S_l$, as one has for $k \neq l$,
\begin{equation}\label{eq:ratio_smoothness}
\frac{\beta_l}{\beta_k} =  
\begin{cases}
1 + \mathcal{O}(\Delta x^3) & \text{if } k=0,l=2 \text{ or } k=2,l=0,\\
1 + \mathcal{O}(\Delta x^2) & \text{otherwise},
\end{cases} 
\end{equation}
which follows from the Taylor expansions of the smoothness indicators \eqref{eq:smoothness_taylor} provided $S_k$ and $S_l$ are smooth with no critical points. Therefore, the ratios of the nonlinear weights satisfy
\begin{equation}\label{eq:proportions}
\frac{\omega_k^\mathrm{JS}}{\omega_l^\mathrm{JS}}  = \frac{\gamma_k}{\gamma_l}  \left( 1+ \mathcal{O}\left(\Delta x^s\right) \right),
\end{equation}
with $s \geq 2$, provided $\varepsilon$ is much smaller than $\beta_k$ and $\beta_l$.\par
A similar computation for the WENO-Z weights, indicated with a superscript $\mathrm{Z}$, shows that this relation also holds, i.e.,
\begin{equation}
\frac{\omega_k^Z}{\omega_l^Z} = \frac{\gamma_k \left(1 + (\frac{\tau}{\beta_k})^p \right) }{\gamma_l \left(1 + (\frac{\tau}{\beta_l})^p \right) } = \frac{\gamma_k \left(\beta_l^p + \tau^p (\frac{\beta_l}{\beta_k})^p \right) }{\gamma_l \left(\beta_l^p + \tau^p \right) }.
\end{equation}
Again using \eqref{eq:ratio_smoothness}, we find that now independent of the value of $\tau$,
\begin{equation}\label{eq:proportions_Z}
\frac{\omega_k^Z}{\omega_l^Z} =  \frac{\gamma_k}{\gamma_l} \left(1 + \mathcal{O}\left(\Delta x^s\right)\right) ,
\end{equation}
with $s \geq 2$ again with $S_k$ and $S_l$ smooth with no critical points. Once more, this relation only depends on the local smoothness of $S_k$ and $S_l$. Note that when the entire stencil $S$ is smooth, the lower bound is increased to $s \geq 3p$.\par
Now consider $S_0$ and $S_1$ being smooth with no critical points, but $S_2$ contains a discontinuity. In this case, both JS and Z schemes will result in $\frac{\omega_0}{\omega_1} = \frac{\gamma_0}{\gamma_1} + \mathcal{O}(\Delta x^2)$. This leads to $\omega_0 \approx \frac{1}{7}$ and $\omega_1 \approx \frac{6}{7}$ for both JS and Z schemes, from which the WENO approximation becomes
\begin{equation}
u_{j+\frac{1}{2}}^\mathrm{(WENO)} - u(x_{j+\frac{1}{2}}) = \tfrac{1}{28} u^{(3)}_j \Delta x^3 + \mathcal{O}(\Delta x^4),
\end{equation}
by a Taylor expansion of the third-order approximations. However, this may not be the optimal choice of weights, as it leads to a third-order linear combination while, for instance, a fourth-order combination is possible if $\omega_0 \approx \frac{1}{4}$ and $\omega_1 \approx \frac{3}{4}$ for this situation.\par 
Relations \eqref{eq:proportions} and \eqref{eq:proportions_Z} therefore show the flaw that is addressed by this work: when the large stencil $S$ is not smooth, WENO-JS and WENO-Z immediately revert to lowest-order modes, even when there are multiple adjacent smooth substencils. In such cases, being able to choose the resulting linear combination has some advantages. If anything, it allows for more control over the numerical solution. More control over the numerical solution in this case means reducing dissipation and increasing resolution.\par
We propose a technique that allows for a choice of the resulting nonlinear weights in the situation when either $\beta_0 = \mathcal{O}(1)$ or $\beta_2 = \mathcal{O}(1)$ and the other substencils are smooth. Consequently, this allows for direct control over the truncation error of the numerical solution in these situations. We call this new type of scheme an embedded WENO scheme. Similarly to conventional WENO schemes, fifth-order accuracy is demanded whenever the numerical solution is smooth on the entire stencil $S$. Moreover, it should reduce to an ENO scheme when two out of three substencils contain a discontinuity.\par
Let us set the question of how to achieve this aside for the moment and first introduce some terminology. The overall third-to-fifth-order accurate scheme is called the outer scheme. The resulting scheme when there are only two adjacent smooth substencils is called the inner scheme, see Figure~\ref{fig:outer_inner_sketch}. For instance, an obvious choice is a fourth-order inner scheme in combination with WENO-JS as the outer scheme.\par
Examining Figure~\ref{fig:outer_inner_sketch} more closely, it becomes clear that if $S_2$ contains the discontinuity and $S_0$ and $S_1$ are smooth, then the discontinuity must lie in the interval $(x_{j+1}, x_{j+2})$. Consequently, there are four grid points on which there is a smooth solution to interpolate. From the two remaining substencils, a four-point stencil can be constructed where the inner scheme is defined. When $S_2$ contains the discontinuity, the available four-point stencil is $S_{0,1} := S_0 \cup S_1$. When $S_0$ contains the discontinuity, the four-point stencil is $S_{1,2} := S_1 \cup S_2$ to use for the inner scheme. \par

\begin{figure}
\centering
\includegraphics[width = .7\textwidth]{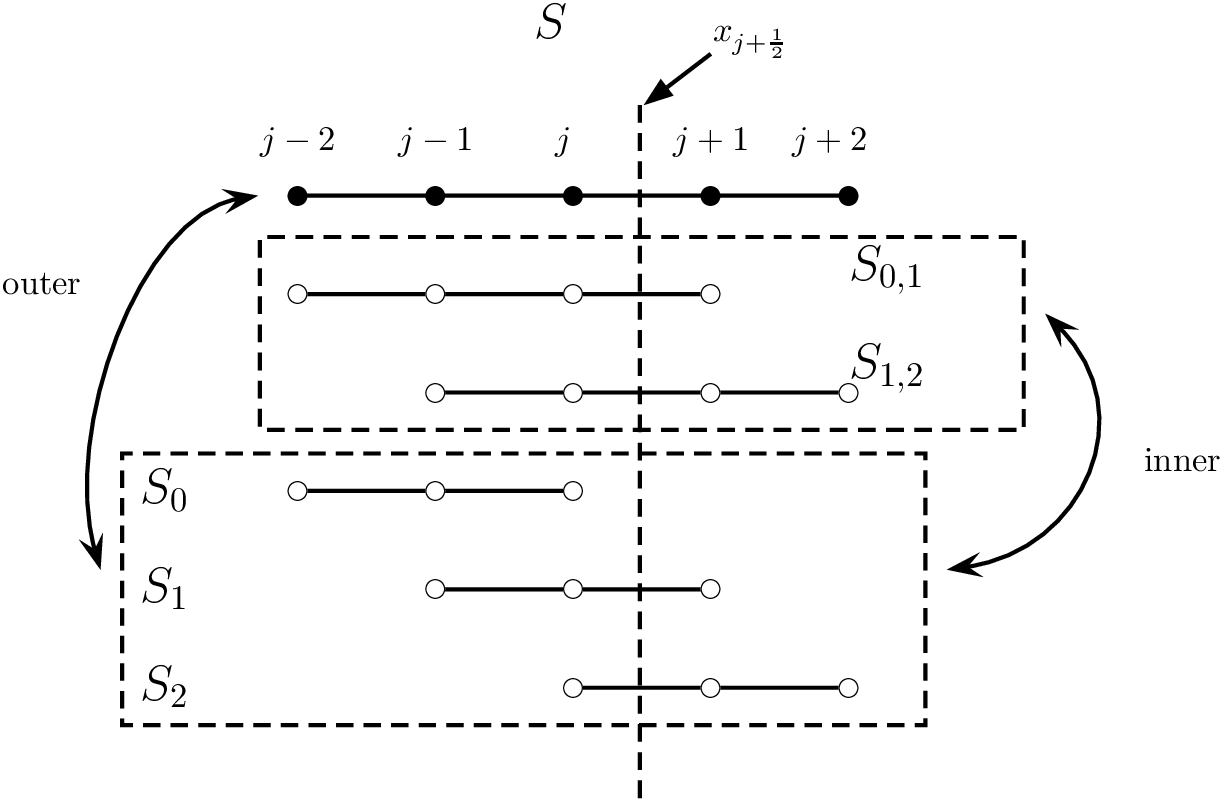}
\caption{The five-point stencil $S$, with substencils $S_0$, $S_1$ and $S_2$, and inner scheme stencils $S_{0,1}$ and $S_{1,2}$.}
\label{fig:outer_inner_sketch}
\end{figure}

Even though a higher formal order of convergence may be obtained, Banks et al. \cite{banks} have pointed out that one often obtains sublinear convergence near linearly degenerate discontinuities, such as the contact waves of the Euler equations. They estimate that the convergence rate becomes $\frac{m}{m+1}$ for a scheme with formal convergence rate $m$. In our case, this suggests the convergence rate is increased from $\frac{3}{4}$ to $\frac{4}{5}$. Thus, the benefits might be less great as a naive estimate would suggest. However, aside from the increased convergence rate, we will also show how embedded schemes can be used to improve spectral properties.\par
With the terminology in place we can turn to the basic question: how to embed one WENO scheme into another. Thus, we would like the nonlinear weights to converge to the inner scheme whenever there are two adjacent smooth substencils and the third one is not smooth. Otherwise, they should remain approximately equal to the nonlinear weights of the outer scheme. This suggests that we multiply the unnormalised weights $\tilde{\omega}_k$ of the outer scheme by a correction that is ordinarily close to unity, but activates when either $\beta_0 $ or $\beta_2$ becomes $\mathcal{O}(1)$. The correction is constructed such that it adjusts the proportions found in \eqref{eq:proportions} and \eqref{eq:proportions_Z}.\par
Suppose the inner scheme is given by the linear weights $\alpha_0^{(2)}$, $\alpha_{1}^{(2)}$, $\alpha_{1}^{(0)}$ and $\alpha_2^{(0)}$. We write the stencils containing a discontinuity in parenthesis in the superscript and the substencil index in the subscript. The desired convex combination then becomes
\begin{subequations}\label{eq:inner_scheme}
\begin{align}
u^{(0,1)}_{j+\frac{1}{2}} &:= \alpha_0^{(2)} u^{(0)}_{j+\frac{1}{2}} + \alpha_{1}^{(2)} u^{(1)}_{j+\frac{1}{2}}, \\
u^{(1,2)}_{j+\frac{1}{2}} &:= \alpha_{1}^{(0)} u^{(1)}_{j+\frac{1}{2}} + \alpha_{2}^{(0)} u^{(2)}_{j+\frac{1}{2}}.
\end{align}
\end{subequations}
We consider two possible choices for the linear weights of the inner scheme, see Table \ref{tab:inner_scheme_weights}. The first is the fourth-order linear combination which is possible on the four-point stencil. The second choice consists of placing the superfluous weight onto the middle substencil, i.e., using the approximation $u^{(k)}_{j+\frac{1}{2}} \approx u^{(1)}_{j+\frac{1}{2}}$ for $k = 0,2$. The fourth-order choice is motivated from an order-of-convergence perspective, while the third-order choice comes from a spectral point of view, see Section \ref{sec:spectral}.

\begin{table}[h]
\centering
\caption{Possible choices for the inner scheme.}
\label{tab:inner_scheme_weights}
\begin{tabular}{l|ll}
                 & 4th            & 3rd             \\ \hline  \rule{0pt}{\normalbaselineskip}
$\alpha_0^{(2)}$       & $\tfrac{1}{4}$ & $\tfrac{1}{10}$ \\[.2cm]
$\alpha_{1}^{(2)}$ & $\tfrac{3}{4}$ & $\tfrac{9}{10}$ \\[.2cm]
$\alpha_{1}^{(0)}$   & $\tfrac{1}{2}$ & $\tfrac{7}{10}$ \\[.2cm]
$\alpha_2^{(0)}$       & $\tfrac{1}{2}$ & $\tfrac{3}{10}$
\end{tabular}
\end{table}

The nonlinear weights must at all times sum to unity to ensure consistency. Thus, any correction we introduce must be incorporated into the unnormalised nonlinear weights and still work after normalisation.\par 
Furthermore, what is happening in substencil $S_2$ must influence both substencils $S_0$ and $S_1$ and \textit{mutatis mutandis} substencil $S_0$ must influence both $S_1$ and $S_2$. It follows that the corrections must be functions of multiple smoothness indicators and thus enforce that the resulting WENO scheme is nonlocal. As a final note, we have seen from \eqref{eq:proportions} and \eqref{eq:proportions_Z} that the nonlinear weights are simply a redistribution of the linear weights. We therefore have to influence the proportions of the linear weights, hence the linear weights should be multiplied with some relative proportions $c_2$ and $c_0$. The role of the relative proportions will become clear later, they are defined as
\begin{subequations}\label{eq:fourth_order_proportions}
\begin{align}
\alpha_0^{(2)} : \alpha_{1}^{(2)} &= c_2 \gamma_0 : \gamma_1,\\
\alpha_{1}^{(0)} : \alpha_{2}^{(0)} &= \gamma_1 : c_0 \gamma_2.
\end{align}
\end{subequations}
The naming convention is again to label the relative proportions with the index of the substencil that is not smooth. We can thus compute the relative proportions using the inner weights suggested in Table \ref{tab:inner_scheme_weights}, see Table \ref{tab:relative_proportions}.

\begin{table}[h]
\centering
\caption{Relative proportions for the 4th order and 3rd order inner schemes. The 3rd order inner scheme places the superfluous weight on the middle stencil.}
\label{tab:relative_proportions}
\begin{tabular}{c|ccc}
     & 4th & 3rd    \\ \hline  \rule{0pt}{\normalbaselineskip} 
$c_2$ & 2 & $\frac{2}{3}$      \\[.2cm]
$c_0$ & 2 & $\frac{6}{7}$              
\end{tabular}
\end{table}

We shall now briefly summarize the conditions that should be satisfied by an embedding correction.

\begin{enumerate}
\item \textit{(Implementation)} The unnormalised nonlinear weights must be multiplied with a correction.

\item \textit{(Nonlocality)} The corrections cannot be functions of only the local smoothness indicators.

\item \textit{(Consistency)} Wherever the solution is smooth on the full stencil, the embedded scheme must reproduce the original scheme.

\item \textit{(Embedding)} When there is a discontinuity present, the scheme must produce the inner weights on the smooth substencils.
\end{enumerate}

\subsection{General framework}
Here, we construct a general framework for embedded WENO schemes. The implementation and consistency conditions suggest that our correction is ordinarily close to a constant, while according to the nonlocality condition it may not be a function of a single smoothness indicator. Therefore, we propose using a general starting point given by
\begin{equation}\label{eq:EW_general_form}
\tilde{\omega}_k^{(\mathrm{E})} = \tilde{\omega}_k^{(\mathrm{O})} \left( a_{kk} + \sum_{l\neq k} \frac{ a_{kl} \beta_l }{\beta_k + \varepsilon} \right),
\end{equation}
where the outer scheme is denoted with superscript $(\mathrm{O})$ and the embedded scheme with $(\mathrm{E})$. This is probably the simplest possible nonlocal correction: a linear combination of ratios. Here, $a_{kl}$ ($k$ and $l$ in the range $0,1,2$) are a collection of undetermined coefficients and $\varepsilon$ is small constant to avoid division by zero. We shall refer to \eqref{eq:EW_general_form} as the general form of an embedded WENO scheme and the term in parenthesis as the general form of a correction.\par
The consistency condition will give us a set of equations that has to be satisfied by the coefficients $a_{kl}$. It tells us that when the solution is smooth all corrections must be close to $1$. Let us assume that the outer scheme satisfies, whenever the solution is smooth with no critical points, $\omega_i^{(\mathrm{O})} = \gamma_i + \mathcal{O}(\Delta x^q)$, then the corrections must satisfy
\begin{equation}\label{eq:consistency_condition}
a_{kk} + \sum_{l \neq k} \frac{ a_{kl} \beta_l }{\beta_k + \varepsilon} = 1 + \mathcal{O}(\Delta x^q),
\end{equation}
which must hold for all $k = 0,1,2$. Using \eqref{eq:ratio_smoothness}, and assuming $\beta_k = \mathcal{O}(\Delta x^2)$ for all $k=0,\ldots,r-1$, where $r$ is the number of substencils, this yields
\begin{equation}\label{eq:consistency_base}
 \sum_{l=0}^{r-1} a_{kl} = 1, \quad k = 0,1,\ldots,r-1.
\end{equation}
If $q = 2$, this is sufficient to satisfy the consistency condition. If $q>2$, for instance for WENO-Z, the coefficients $a_{kl}$ must also provide linear combinations of smoothness indicators that cancel out the lower order terms in the Taylor expansions \eqref{eq:smoothness_taylor}.\par 
To achieve this, the general form \eqref{eq:EW_general_form} is adjusted. In this case, at least the first term in the error expansion of \eqref{eq:consistency_condition} must vanish. The constant term in the correction must still equal 1, which suggests we adjust the general form to read
\begin{equation}\label{eq:EW_general_form2}
\tilde{\omega}_k^{(\mathrm{E})} = \tilde{\omega}_k^{(\mathrm{O})} \left(1 +   \left( \frac{ \left| \sum_{l =0 }^{r-1}  a_{kl} \beta_l  \right| }{\beta_k + \varepsilon}  \right)^p \right),
\end{equation}
where $r$ is again the number of substencils, $p\geq 1$ and $\varepsilon$ is a small constant to avoid division by $0$. We shall refer to \eqref{eq:EW_general_form2} as the second general form. Setting $\tilde{\omega}_k^\mathrm{(O)} = \gamma_k$, it becomes clear this form can be considered as a generalisation of the WENO-Z weights \eqref{eq:WENO-Z_weights}. Here, at least the lowest-order term from the smoothness indicators \eqref{eq:smoothness_taylor} must vanish, leading to
\begin{equation}\label{eq:consistency_base2}
 \sum_{l=0}^{r-1} a_{kl} = 0, \quad k=0,1,\ldots,r-1.
\end{equation}
Regardless of which general form is chosen, \eqref{eq:EW_general_form} or \eqref{eq:EW_general_form2}, further conditions on the coefficients $a_{kl}$ are obtained by the embedding condition. These can be derived by examining the possible positions of a discontinuity and setting the resulting weights equal to the inner weights.\par 
Although throughout this work we have restricted ourselves to five-point WENO schemes, the conditions \eqref{eq:consistency_base} and \eqref{eq:consistency_base2} holds for any number of substencils $r$. In deriving the embedding equations, we shall also take a more general view.

\begin{thm} \label{thm:embedding_equations} (Embedding equations) Let $\tilde{\omega}_k$ be the unnormalised nonlinear weights of a WENO scheme that has $r$ substencils and satisfies $\frac{\tilde{\omega}_k}{\tilde{\omega}_l} \to \frac{\gamma_k}{\gamma_l}$ as $\Delta x \to 0$ whenever $S_k$ and $S_l$ are smooth. Let the unnormalised embedded WENO weights be given by
\begin{equation}
\tilde{\omega}_k^{(\mathrm{E})} = \tilde{\omega}_k g_k, \quad g_k = a_{kk} + \sum\limits_{l\neq k} \frac{a_{kl} \beta_l}{\beta_k + \varepsilon}, \quad k=0,1,2,\ldots,r-1,
\end{equation}
where $g_k$ is the correction factor from the first general form \eqref{eq:EW_general_form}. Let $K$ be the set of indices such that $\beta_n = \mathcal{O}(1)$ for $n \in K$, i.e. $S_n$ is not smooth, and let $\beta_k \downarrow 0$ for $k \not \in K$ as $\Delta x \to 0$. Then, the embedding equations are given by
\begin{equation}\label{eq:embedding_equations}
 \frac{\gamma_k}{\alpha_k^{(K)}}  \sum_{m \in K} a_{km} =  \frac{\gamma_l}{\alpha_l^{(K)}}  \sum_{n \in K} a_{ln},
\end{equation}
with $k,l \not \in K$, $n,m \in K$. Here $\alpha^{(K)}_k$ are the desired inner weights and $\gamma_k$ the linear weights.
\end{thm}

\begin{proof}
1. Fix some set $K$ and assume that $\varepsilon$ is so small it may be ignored in the analysis. Let $k,l \not \in K$, then the ratio of two embedded weights is given by
\begin{equation*}
\frac{\omega_k^\mathrm{(E)}}{\omega_l^\mathrm{(E)}} = \frac{\tilde{\omega}_k^\mathrm{(E)}}{\tilde{\omega}_l^\mathrm{(E)}} = \frac{\tilde{\omega}_k}{\tilde{\omega}_l} \frac{g_k}{g_l}. \tag{$\ast$}
\end{equation*}
2. For $k \not \in  K$, the correction becomes
\begin{equation*}
g_k = \sum_{n \not \in K} a_{kn} + \sum_{n \in K} a_{kn} \frac{\beta_n}{\beta_k} + \mathcal{O}(\Delta x^s) ,
\end{equation*}
where $s \geq 2$ from \eqref{eq:ratio_smoothness}. Next, we use that $\beta_n = \mathcal{O}(1)$ for $n \in K$, so that we obtain
\begin{equation*}
g_k = \frac{C}{\Delta x^2} \sum_{n \in K} a_{kn} + \mathcal{O}(1),  \tag{$\star$}
\end{equation*}
where $C$ is a constant.\\
3. Next, $(\star)$ is substituted into $(\ast)$ leading to
\begin{equation*}
 \frac{\tilde{\omega}_k^\mathrm{(E)}}{\tilde{\omega}_l^\mathrm{(E)}} = \frac{\tilde{\omega}_k}{\tilde{\omega}_l} \frac{ \sum_{n \in K} a_{kn} + \mathcal{O}(\Delta x^2) }{ \sum_{m \in K} a_{lm} + \mathcal{O}(\Delta x^2) }.
\end{equation*}
We let $\Delta x \downarrow 0$ so that $\frac{\tilde{\omega}_k}{\tilde{\omega}_l} \to \frac{\gamma_k}{\gamma_l}$, therefore
\begin{equation*}
\frac{\omega_k^\mathrm{(E)}}{\omega_l^\mathrm{(E)}} =   \frac{\gamma_k}{\gamma_l} \frac{\sum_{n \in K} a_{kn}   }{ \sum_{m \in K} a_{lm} }.
\end{equation*}
This ratio has to be equal to the ratio of inner weights, so that
\begin{equation*}
\frac{\alpha^\mathrm{(K)}_k}{\alpha^\mathrm{(K)}_l}=   \frac{\gamma_k}{\gamma_l} \frac{\sum_{n \in K} a_{kn}   }{ \sum_{m \in K} a_{lm} },
\end{equation*}
which can be simplified to \eqref{eq:embedding_equations}.  \hfill$\square$
\end{proof} 

\begin{remark} The assumption that the outer weights should satisfy $\frac{\tilde{\omega}_k}{\tilde{\omega}_l} \to \frac{\gamma_k}{\gamma_l}$ as $\Delta x \to 0$ includes the choices of JS weights, Z weights and simply using the linear weights $\tilde{\omega}_k  = \gamma_k$.
\end{remark}

\begin{thm} \label{thm:embedding_equations_form2} Under the same assumptions as Theorem~\ref{thm:embedding_equations} and using the second general form \eqref{eq:EW_general_form2}, the embedding equations are given by
\begin{equation}\label{eq:embedding_equations_form2}
\left( \frac{\gamma_k}{\alpha_k^{(K)}} \right)^{\frac{1}{p}} \sum_{m \in K} a_{km} = \pm \left( \frac{\gamma_l}{\alpha_l^{(K)}} \right)^{\frac{1}{p}} \sum_{n \in K} a_{ln},
\end{equation}
\end{thm}

\begin{proof}
1. Repeating the steps of the previous proof with $k \not \in K$, the correction for the second form now equals
\begin{equation*}
g_k = 1+ \left| \frac{C}{\Delta x^2} \sum_{n \in K} a_{kn} + \sum_{n \not \in K} a_{kn} \left( 1+ \mathcal{O}(\Delta x^2) \right)   \right|^p,
\end{equation*}
so that
\begin{equation*}
\frac{g_k}{g_l} = \frac{| \sum_{n \in K} a_{kn} + \mathcal{O} (\Delta x^2) |^p }{| \sum_{m \in K} a_{lm} + \mathcal{O} (\Delta x^2) |^p}.
\end{equation*}
2. Passing to the limit $\Delta x \downarrow 0$, we find
\begin{equation*}
\frac{\omega_k^\mathrm{(E)}}{\omega_l^\mathrm{(E)}} =   \frac{\gamma_k}{\gamma_l} \left| \frac{\sum_{n \in K} a_{kn} }{\sum_{m \in K} a_{lm}}  \right|^p.
\end{equation*}
Setting the left-hand side equal to $\frac{\alpha^\mathrm{(K)}_k}{\alpha^\mathrm{(K)}_l}$, this yields
\begin{equation*}
\left( \frac{\gamma_k}{\alpha^\mathrm{(K)}_k} \right)^\frac{1}{p} \left| \sum_{n \in K} a_{kn} \right| = \left( \frac{\gamma_l}{\alpha^\mathrm{(K)}_l} \right)^\frac{1}{p} \left| \sum_{m \in K} a_{lm} \right|.
\end{equation*}
Thus, if the coefficients satisfy \eqref{eq:embedding_equations_form2}, the embedded weights will converge to the inner weights. \hfill$\square$
\end{proof}

\begin{remark}
There is a freedom in the choice of sign of the coefficients for the second form \eqref{eq:EW_general_form2}. Interpreting the coefficients as elements of a matrix $A$, any row may be multiplied with $-1$ with impunity. From here on out, we use the positive sign in \eqref{eq:embedding_equations_form2}.
\end{remark}

The embedding equations are a set of linear equations for the coefficients $a_{kl}$, since the inner weights are given or rather chosen by the user. The embedding equations relate the weights of the inner scheme to the linear weights. Together with the equations coming from the consistency condition, this will provide a number of linear equations for the coefficients $a_{kl}$. For five-point WENO schemes, we find that $K$ can be either $\{0\}$ or $\{2\}$, the other cases being already included in the WENO weights. In each case for $K$ there are only two remaining smooth substencils. We thus end up with two equations
\begin{subequations}
\begin{align}
\left( \frac{\gamma_0 }{\alpha^{(2)}_0} \right)^\frac{1}{p} a_{02} &= \left( \frac{\gamma_1 }{\alpha^{(2)}_1} \right)^\frac{1}{p} a_{12},\\
\left(\frac{\gamma_2 }{\alpha^{(0)}_2}\right)^\frac{1}{p} a_{20} &=\left( \frac{\gamma_1 }{\alpha^{(0)}_1}\right)^\frac{1}{p} a_{10},
\end{align}
\end{subequations}
where the $p=1$ equations also apply to the first form \eqref{eq:EW_general_form}. These may be simplified using our earlier definition of the relative proportions $c_0$ and $c_2$ \eqref{eq:fourth_order_proportions}, i.e.,
\begin{subequations}\label{eq:WENO5_embedding_eqs}
\begin{align}
\frac{a_{02}}{a_{12}} &= (c_2)^\frac{1}{p},\\
\frac{a_{20}}{a_{10}} & = (c_0)^\frac{1}{p},
\end{align}
\end{subequations}
where once again, the $p=1$ equations apply to both forms provided $c_2>0$ and $c_0>0$, for the second form one can use $p>1$.

\section{Implementation}\label{sec:implementation}

\subsection{Embedded WENO-JS}
We will now show how to construct embedded WENO schemes using the WENO-JS scheme as an outer scheme. We will assume the inner weights $\alpha_0^{(2)}$, $\alpha_{1}^{(2)}$, $\alpha_{1}^{(0)}$ and $\alpha_2^{(0)}$ are given, e.g., chosen from Table \ref{tab:inner_scheme_weights}. From the inner weights, we can find their relative proportions as measured against the outer weights by \eqref{eq:fourth_order_proportions}, see Table \ref{tab:relative_proportions}. We shall use the general form \eqref{eq:EW_general_form} as a template. Furthermore, we have that $q = 2$, so that \eqref{eq:consistency_base} provides three equations that are sufficient to ensure that the scheme is unaltered when the solution is smooth. The two embedding equations for a five-point WENO scheme are given by \eqref{eq:WENO5_embedding_eqs}. Hence, we have five equations for nine coefficients that can be solved to yield a four-parameter family of embedded schemes, given by
\begin{subequations}
\begin{align}
a_{00} &= 1-a_{01} - a_{02},\\
a_{11} &= 1- \frac{a_{20}}{c_0} - \frac{a_{02} }{c_2},\\
a_{22} &= 1- a_{20} - a_{21},\\
a_{12} &= \frac{a_{02}}{c_2},\\
a_{10} &=\frac{a_{20}}{c_0},
\end{align}
\end{subequations}
where $a_{01}$, $a_{02}$, $a_{20}$ and $a_{21}$ can be chosen freely. We have experimented with a number of possible choices, all seemed to provide improvements over the WENO-JS scheme. However, different choices resulted in schemes with different behaviour, much like choosing a different flux limiter in a TVD scheme.\par
We shall continue with the embedded scheme that appears to have the best all-round performance, it can be constructed using the choices $a_{01}=a_{21}=0$, $a_{20} = \frac{c_0}{3}$ and $a_{02} = \frac{c_2}{3}$. For this particular choice of embedded WENO, we may even choose the linear weights as the outer scheme, such that we obtain
	\begin{subequations}\label{eq:final_form}
		\begin{align}
		\tilde{\omega}_0 &= \tfrac{1}{3}  \gamma_0 \left( 3-c_2 + c_2 \frac{\beta_2}{\beta_0 + \varepsilon}  \right),\\
		\tilde{\omega}_1 &= \tfrac{1}{3} \gamma_1 \left( 1 +  \frac{\beta_2}{\beta_1 + \varepsilon} + \frac{\beta_0}{\beta_1 + \varepsilon} \right),\\
		\tilde{\omega}_2 &=  \tfrac{1}{3} \gamma_2 \left( 3-c_0 + c_0\frac{\beta_0}{\beta_2 + \varepsilon} \right).
		\end{align}
	\end{subequations}
This scheme yields a convex combination when all weights are positive, thus we must have $c_0<3$ and $c_2<3$, which includes the choices presented in Table \ref{tab:relative_proportions}.\par
To show that we may use this choice, we apply a Taylor expansion to the normalised weights, under the assumption of smooth solutions without critical points. A lengthy computation or symbolic calculation will show that the weights from \eqref{eq:final_form} satisfy
\begin{equation}
\omega_k - \gamma_k = \tfrac{1}{3} \left( \omega_k^\mathrm{JS} - \gamma_k \right) + \mathcal{O}(\Delta x^3),
\end{equation}
regardless of the values of $c_0$ and $c_2$. Therefore, the embedded WENO scheme \eqref{eq:final_form} also provides fifth-order convergence. Under the assumption that only one substencil is smooth, the weights \eqref{eq:final_form} also provide the proper behaviour. Indeed, fix $k$ and set $\beta_k = \mathcal{O}(\Delta x^2)$ and $\beta_l = \mathcal{O}(1)$ with $l \neq k$, then we find that $\tilde{\omega}_k = \mathcal{O}(\frac{1}{\Delta x^2})$ and $\tilde{\omega}_l = \mathcal{O}(1)$. Therefore, with only one smooth substencil, we find $\omega_k = 1 + \mathcal{O}(\Delta x^2)$ and $\omega_l = \mathcal{O}(\Delta x^2)$.\par
We conclude that the embedded WENO scheme given by \eqref{eq:final_form} is equivalent to the standard WENO-JS scheme for smooth solutions without critical points or having only a single smooth substencil. When there are two adjacent smooth substencils and the third one is not smooth, we obtain the inner scheme.\par
To verify that the embedded schemes have the same order of convergence for smooth functions, a short test is performed. In \cite{borges}, it is shown that the conservative difference in \eqref{eq:finite_volume} can be interpreted as a differentiation operator if applied to the original function instead of its averages. WENO-JS only features optimal convergence for smooth functions without critical points, therefore the first test consists of applying WENO differentiation to a test function given by
\begin{equation}\label{eq:tanh}
u_1(x) = \tanh\left( 10 x \right).
\end{equation}
The boundary conditions are supplied exactly using fictitious grid points. For functions featuring first-order critical points, WENO-JS provides fourth-order accuracy. Therefore, the second test function is given by
\begin{equation}\label{eq:sin_in_sin}
u_2(x) = \sin \left( \pi x - \frac{\sin(\pi x)}{\pi}  \right).
\end{equation}
In both tests, the error is computed using the scaled absolute sum
\begin{equation}
e = \sum_{j=1}^N \big|D u_j - u^\prime(x) \big| \Delta x,
\end{equation}
where $D$ is the WENO differentiation operator. The parameter $\varepsilon$ is set to $10^{-40}$ and $c_2 = c_0=2$ in this example. The number of grid points $N$ are chosen such that the grid size $\Delta x$ is halved each time. The results are given in Table~\ref{tab:convergence_tanh_EJS} and clearly demonstrate optimal convergence for smooth functions with no critical points and fourth-order convergence for smooth functions with first-order critical points. Thus, the embedded WENO-JS scheme provides the same or similar performance as the original scheme for smooth functions.

\begin{table}
	\centering
		\caption{Convergence test for the embedded WENO-JS scheme using \eqref{eq:tanh} with $c_2= c_0=2$ and $\varepsilon=10^{-40}$.}
		\label{tab:convergence_tanh_EJS}
		\begin{tabular}{l|ll|ll}
			$N$      & error$_1$                & order$_1$ & error$_2$                    & order$_2$  \\	 \hline 
			$101$  & $2.0 \cdot 10^{-4}$ &                    &	$1.6 \cdot 10^{-6}$&  						  \\
			$201$ & $7.1 \cdot 10^{-6}$ & $4.8$      &	$7.2 \cdot 10^{-8}$&  	$4.5$		   \\
			$401$ & $2.3 \cdot 10^{-7}$ & $4.9$      &	$3.8 \cdot 10^{-9}$ &  $4.2$		   \\
			$801$ & $7.3 \cdot 10^{-9}$  & $5.0$       &$2.1 \cdot 10^{-10}$  & $4.2$			  \\
			$1601$&$2.3 \cdot 10^{-10}$ & $5.0$      &$1.3 \cdot 10^{-11}$ & $4.0$			  
		\end{tabular}
\end{table}

\subsection{Embedded WENO-Z}
A more contemporary version of WENO schemes is represented by the WENO-Z scheme of Borges et al. \cite{borges}. As mentioned earlier, the WENO-JS scheme has the property that $\omega_k = \gamma_k + \mathcal{O}(\Delta x^2)$ for smooth solutions, whereas the WENO-Z weights satisfy $\omega_k^\mathrm{Z} = \gamma_k + \mathcal{O}(\Delta x^{3p})$, with $p$ the power parameter. Consequently, at critical points, the WENO-Z scheme avoids loss of convergence. A side-effect of the new weights is faster convergence to the linear weights in smooth regions. This also results in sharper resolution of discontinuities. The unnormalissed weights for the WENO-Z scheme are defined as in \eqref{eq:WENO-Z_weights}.\par
Embedding an inner scheme into the WENO-Z scheme is somewhat easier, since the second general form \eqref{eq:EW_general_form2} is a generalisation of WENO-Z. In the context of our framework, we have to satisfy the consistency conditions \eqref{eq:consistency_base2}, i.e., $\sum_l a_{kl} = 0$ for $k=0,1,2$. At the same time, we can obtain extra equations from \eqref{eq:smoothness_taylor}, where we find the fourth-order term must cancel out as well, i.e.,
\begin{equation}
a_{k0} - \tfrac{1}{2} a_{k1} + a_{k2} = 0, \quad k = 0,1,2.
\end{equation}
By Theorem~\ref{thm:embedding_equations_form2}, the two embedding equations are now given by \eqref{eq:WENO5_embedding_eqs}. Thus, for an embedded version of WENO-Z, we have six equations from consistency and two embedding equations to solve for nine coefficients, yielding a one-parameter family of schemes, given by 
\begin{subequations}\label{eq:WENO-Z_embedded}
\begin{align}
\tilde{\omega}_0 &= \gamma_0 \left( 1 + \mu c_2 \left(\frac{\tau}{\beta_0+\varepsilon} \right)^p  \right),\\
\tilde{\omega}_1 &= \gamma_1 \left( 1 + \mu \left(\frac{\tau}{\beta_1+\varepsilon} \right)^p   \right),\\
\tilde{\omega}_2 &= \gamma_2 \left( 1 + \mu c_0 \left(\frac{\tau}{\beta_2+\varepsilon} \right)^p   \right),
\end{align}
\end{subequations}
where $\mu$ is the free parameter, set to $\frac{1}{4}$ unless mentioned otherwise, and again $\tau = | \beta_0 - \beta_2|$. The scheme given by \eqref{eq:WENO-Z_embedded} is stable for $\mu >0$, $c_0>0$ and $c_2>0$, which includes the options presented in Table \ref{tab:relative_proportions}. The power parameter is set to $p=2$ throughout the rest of this work. \par
As earlier, we perform a convergence test to verify that the embedded WENO-Z scheme has the same order of convergence as its standard counterpart for smooth functions. The details can be found in the previous subsection as the test procedure is exactly the same. WENO-Z with $p=2$ attains optimal convergence for smooth functions that may have first-order critical points. Only third-order accuracy is attained when the order of the critical points is higher. Therefore, only the second test is performed with initial condition \eqref{eq:sin_in_sin}, which features two first-order critical points. We furthermore set $c_2 = c_0=2$ and $\varepsilon = 10^{-40}$. The results are given in Table~\ref{tab:convergence_EZ} and once again show optimal convergence.

\begin{table}[h]
	\centering
	\caption{Convergence test for the embedded WENO-Z scheme using \eqref{eq:sin_in_sin} with $c_2= c_0=2$ and $\varepsilon=10^{-40}$.}
	\label{tab:convergence_EZ}
	\begin{tabular}{l|l|l}
		$N$      & error$_2$              & order$_2$  \\	 \hline 
		$101$  & $6.0 \cdot 10^{-7}$ &      \\
		$201$ & $1.8 \cdot 10^{-8}$ & $5.1$ \\
		$401$ & $5.9 \cdot 10^{-10}$ & $5.0$ \\
		$801$ & $1.8 \cdot 10^{-11}$  & $5.0$ \\
		$1601$&$6.0 \cdot 10^{-13}$ & $5.0$
	\end{tabular}
\end{table}

\subsection{Notation of embedded schemes}
As indicated earlier, the relative proportions $c_2$ and $c_0$ can be chosen independently. Thus, we may choose a fourth-order inner scheme on $S_{1,2}$, while on $S_{0,2}$ we may place the superfluous weight on the middle stencil. To clarify which scheme is being used at a particular time, we propose the following notation. We shall write the outer scheme with the relative proportions in parenthesis: WENO-JS($c_2$, $c_0$) and WENO-Z($c_2$, $c_0$). We write $c_2$ first since it affects the interpolation on the left when the discontinuity is on the right. In special cases, we shall explicitly name a scheme, such as the WENO-JS scheme with fourth-order inner scheme WENO-45 := WENO-JS(2,2).\par
The tableau notation of \eqref{eq:WENO5_tableau} can be extended to include the matrix $A$ with elements $a_{kl}$, i.e.
\begin{equation}
\begin{array}{c|c|c}
C & \boldsymbol{\gamma} & A \\ \hline
\end{array}.
\end{equation}
A tableau together with a form, \eqref{eq:EW_general_form} or \eqref{eq:EW_general_form2}, completely specifies an embedded WENO scheme. As such, the embedded WENO-JS scheme \eqref{eq:final_form} has a tableau given by
\begin{equation}
\begin{array}{ccccc|c|ccc}
\frac{2}{6} & -\frac{7}{6} & \frac{11}{6} & & & \frac{1}{10} & 3-c_2 & 0 & c_2 \\ \rule{0pt}{2.6ex}
& -\frac{1}{6} & \frac{5}{6} & \frac{2}{6} & & \frac{6}{10} & 1& 1 & 1 \\ \rule{0pt}{2.6ex}
& & \frac{2}{6} & \frac{5}{6} & -\frac{1}{6} & \frac{3}{10} & c_0 & 0 & 3-c_0\\ \hline
\end{array}\begin{matrix} \vphantom{x} \\ \vphantom{y} \\ \vphantom{z}.\end{matrix}
\end{equation}
The embedded WENO-Z scheme (with $p=2$) has a tableau
\begin{equation}
\begin{array}{ccccc|c|ccc}
\frac{2}{6} & -\frac{7}{6} & \frac{11}{6} & & & \frac{1}{10} & \frac{\sqrt{c_2}}{2} & 0 & -\frac{\sqrt{c_2}}{2} \\ \rule{0pt}{2.6ex}
& -\frac{1}{6} & \frac{5}{6} & \frac{2}{6} & & \frac{6}{10} & \frac{1}{2} & 0 & -\frac{1}{2} \\ \rule{0pt}{2.6ex}
& & \frac{2}{6} & \frac{5}{6} & -\frac{1}{6} & \frac{3}{10} & \frac{\sqrt{c_0}}{2} & 0 & -\frac{\sqrt{c_0}}{2}\\ \hline
\end{array}\begin{matrix} \vphantom{x} \\ \vphantom{y} \\ \vphantom{z}.\end{matrix}
\end{equation}

\section{Spectral properties}\label{sec:spectral}
The embedded WENO schemes may be investigated by analysing their spectral properties. The inner scheme activates whenever $\beta_0$ or $\beta_2$ become significantly larger than the other two smoothness indicators. In terms of sinusoidal functions, one would expect this to happen in the medium-range of wave numbers. Thus, the spectral properties of a WENO scheme in this regime can be improved by embedding an inner scheme.\par
As an example, we will show that we can reduce the dissipation of a WENO scheme by embedding an inner scheme. This is particularly useful when working with smooth solutions. On the other hand, when working with sharply varying or even discontinuous solutions, it may be desirable to increase dissipation to obtain greater stability.\par
One can investigate the spectral properties of a WENO scheme by analysing the underlying linear schemes \cite{martin}. We can interpret the WENO schemes as a linear combination of the underlying third-order upwind schemes, where the weights vary with the wave numbers. This way, we may find upper and lower bounds for the spectral curves. We consider here plane wave solutions to the linear advection equation,
\begin{equation}\label{eq:linear_advection}
\frac{\partial u}{\partial t} + \frac{\partial u}{\partial x} = 0.
\end{equation}
We can relate the numerical solution obtained by the underlying linear schemes to a plane wave solution $u(x,t) = \exp( \mathrm{i} (\kappa x - \omega t))$. After some manipulation, we find the modified wave number $\kappa^\ast$, which is a complex quantity related to the spectral properties of the scheme. The imaginary value of $\kappa^\star$ determines the dissipation, $\mathrm{Im}( \kappa^\ast)=0$ being a nondissipative scheme. The real part of $\kappa^\ast$ determines the phase error in the numerical approximation, thus relating to dispersion.\par
Let us first investigate the basic spectral properties of the three possible third-order approximations and the fifth-order linear combination, see Figure~\ref{fig:base_modified_wavenumber}. A WENO scheme will give a fifth-order approximation for smooth solutions, while a third-order approximation for rapidly varying solutions. Thus, we expect the WENO scheme to follow the fifth-order curves for low wave numbers and the third-order curves for high wave numbers. These considerations give us a qualitative understanding of WENO methods. The embedded methods will switch to their inner scheme for mid-range wave numbers.

\begin{figure}[h]
\centering
\includegraphics[width=\textwidth]{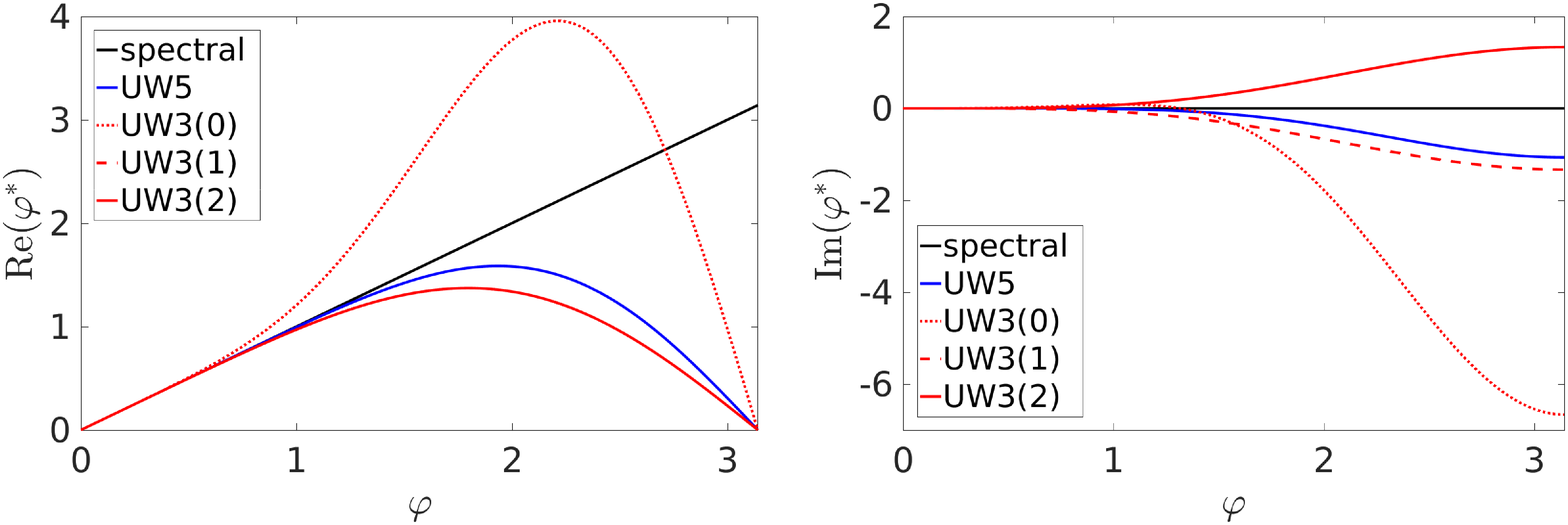}
\caption{Dissipation (left) and dispersion (right) curves for the underlying linear schemes integrated with SSPRK(3,3). UW5 indicates the fifth-order upwind approximation. The label in parenthesis for the UW3 schemes indicate on what stencil it works, thus UW3($i$) works on $S_i$. The dispersion curves for UW3(1) and UW3(2) are equal.}
\label{fig:base_modified_wavenumber}
\end{figure}

Let us now study the inner schemes, which are four-point linear schemes given by \eqref{eq:inner_scheme}, completed by Table \ref{tab:inner_scheme_weights}. The resulting dispersion and dissipation curves are presented in Figure~\ref{fig:inner_modified_wavenumber}. An important thing to note is that all inner schemes do not support parasitic wave modes, since $\mathrm{Im}( \varphi^\ast) \leq 0$ across the whole range.

\begin{figure}[h]
\centering
\includegraphics[width=\textwidth]{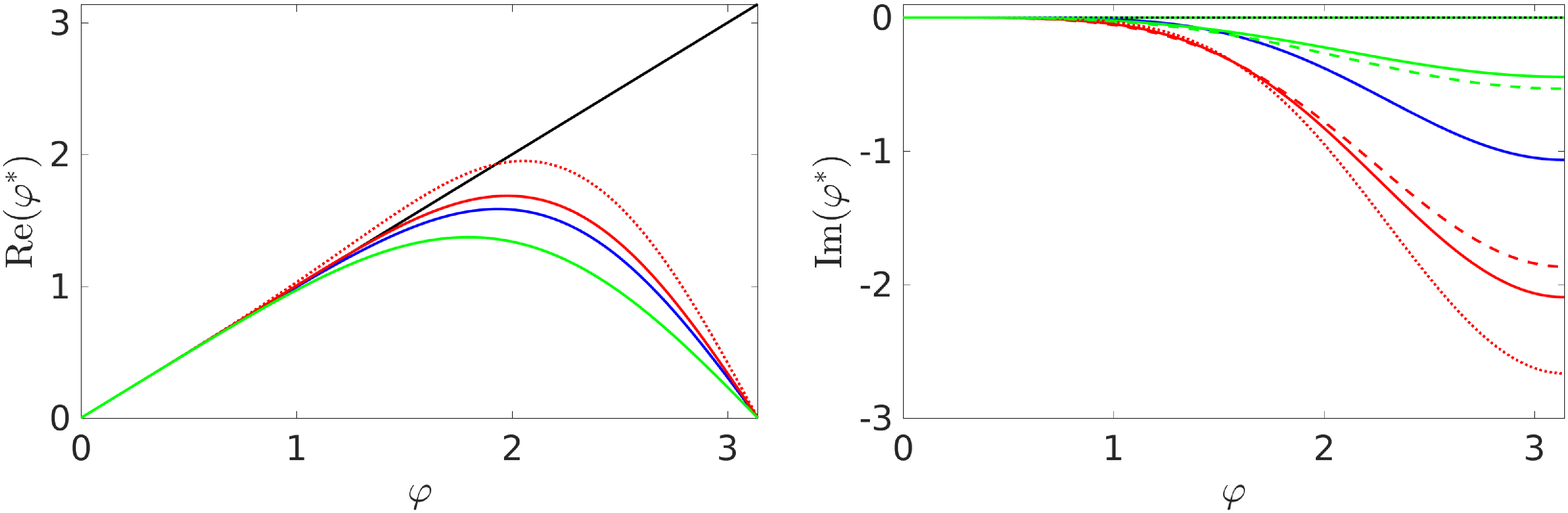}
\caption{Dissipation (left) and dispersion (right) curves for the inner scheme. The blue curves are the fifth-order curves, the red curves are schemes working on $S_0 \cup S_1$ and the green curves work on $S_1 \cup S_2$. The solid curves represent $c_0 = c_2 = 1$, the dashed curves represent $(c_2,c_0) = (\tfrac{2}{3},\tfrac{6}{7})$ and the dotted curves represent $c_0 = c_2 = 2$.}
\label{fig:inner_modified_wavenumber}
\end{figure}

What becomes clear from the curves is that one can certainly influence the spectral properties of the scheme. Also, it should be noted that the curve corresponding to $c_2 = \tfrac{2}{3}$ gives a dispersion curve which is equal to the fifth-order dispersion curve. Therefore, this justifies this particular choice, as it reduces the phase difference near discontinuities.\par
The influence of inner schemes is demonstrated by some numerical examples. We solve the linear advection equation~\eqref{eq:linear_advection} using several variants of WENO schemes on a periodic domain with the initial condition a complex exponential plane wave. We pick as the wave number $\kappa = 10 \pi$, thus the initial condition is given by
\begin{equation}
u_0(x) = e^{i 10 \pi x},
\end{equation}
with the computational domain $-1 \leq x \leq 1$. As the initial condition is smooth, we shall use the inner scheme to reduce dissipation. Hence, we shall compare the standard WENO-JS and WENO-Z schemes to the WENO-JS($\tfrac{2}{3}$,$2$) and WENO-Z($\tfrac{2}{3}$,$2$) variants. For the embedded WENO-Z scheme, we set the free parameter $\mu = \frac{1}{4}$ in \eqref{eq:WENO-Z_embedded}. The advection equation is integrated for 64 time units using the SSPRK(3,3) method and examine the amplitude of the numerical solutions, see Figure~\ref{fig:sine_wave_period10}. For all these computations, $\varepsilon = 10^{-12}$.

\begin{figure}[h]
\centering
\includegraphics[width = .9\textwidth]{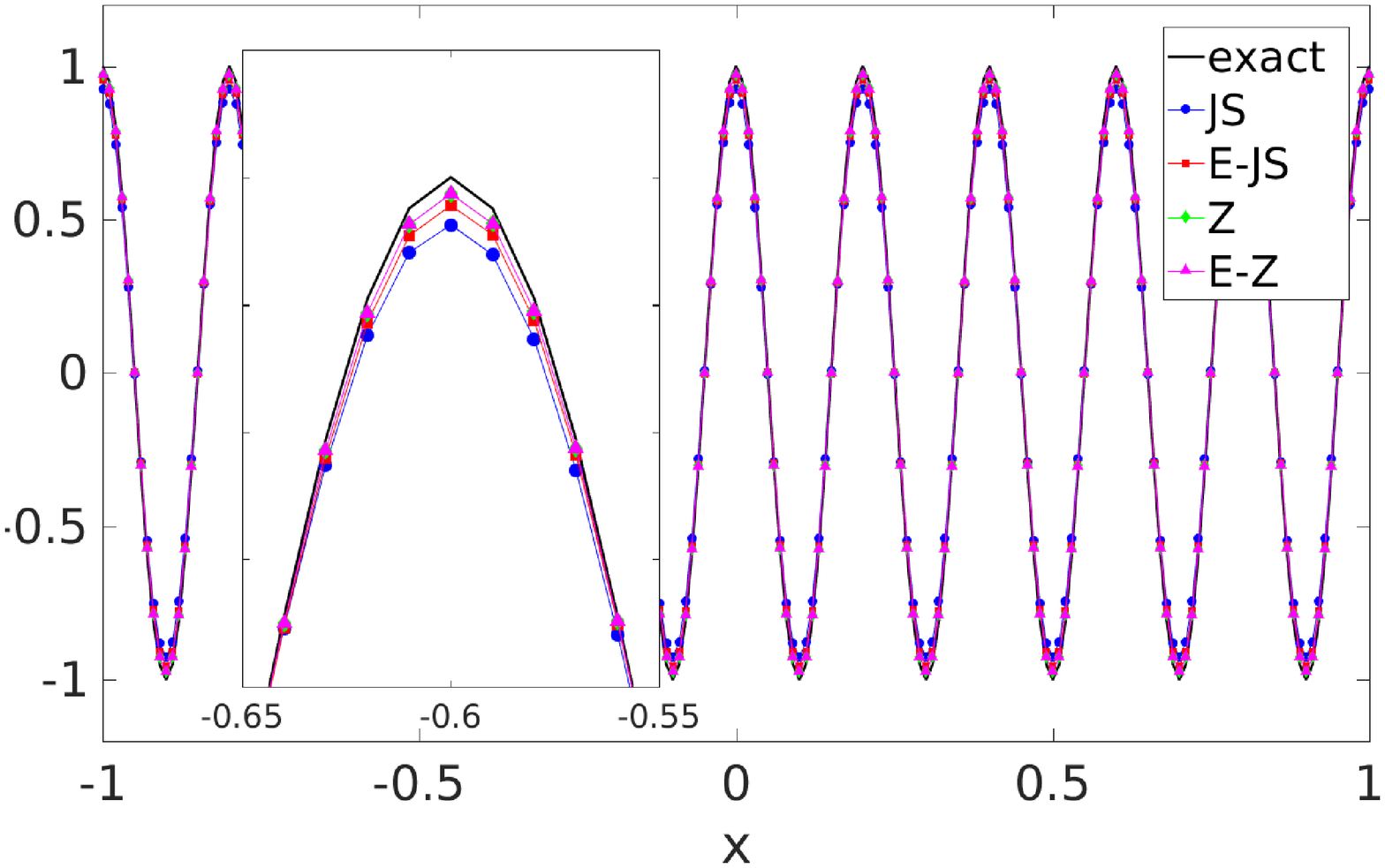}
\caption{Real part of the complex exponential plane wave with wave number $\kappa = 10 \pi$ at $t=8$ on a periodic domain with 201 grid points and a CFL number of $0.5$, resulting in $800$ time steps.The WENO-Z solution is equal to the embedded WENO-Z solution.}
\label{fig:sine_wave_period10}
\end{figure}

All the WENO schemes are, in this case, still solving with fifth-order accuracy in most of the domain. For this wave number, dispersive effects are at a minimum for all schemes. Dissipative effects are noticeable, though small, with WENO-JS being the most dissipative. Embedded WENO-JS($\frac{2}{3}$,$2$) clearly has less dissipation as intended. WENO-Z and its embedded version have the same amount of dissipation, being both the least dissipative of the four presented schemes.\par
We will now investigate a higher wave number. We use the initial condition
\begin{equation}
u_0(x) = e^{i 20 \pi x},
\end{equation}

The results are plotted in Figure~\ref{fig:sine_wave_period20}. Now, the dissipative effects are more pronounced, with WENO-JS having damped out the wave significantly. The embedded WENO-JS($\frac{2}{3}$,$2$) scheme clearly has less dissipation than its standard counterpart. For WENO-Z, nonlinear effects are starting to become clear, as different parts of the domain are affected differently. For the embedded version, the wave has kept its form much better, i.e., there is less dissipation and the dispersion error is smaller.

\begin{figure}[h]
\centering
\includegraphics[width = .9\textwidth]{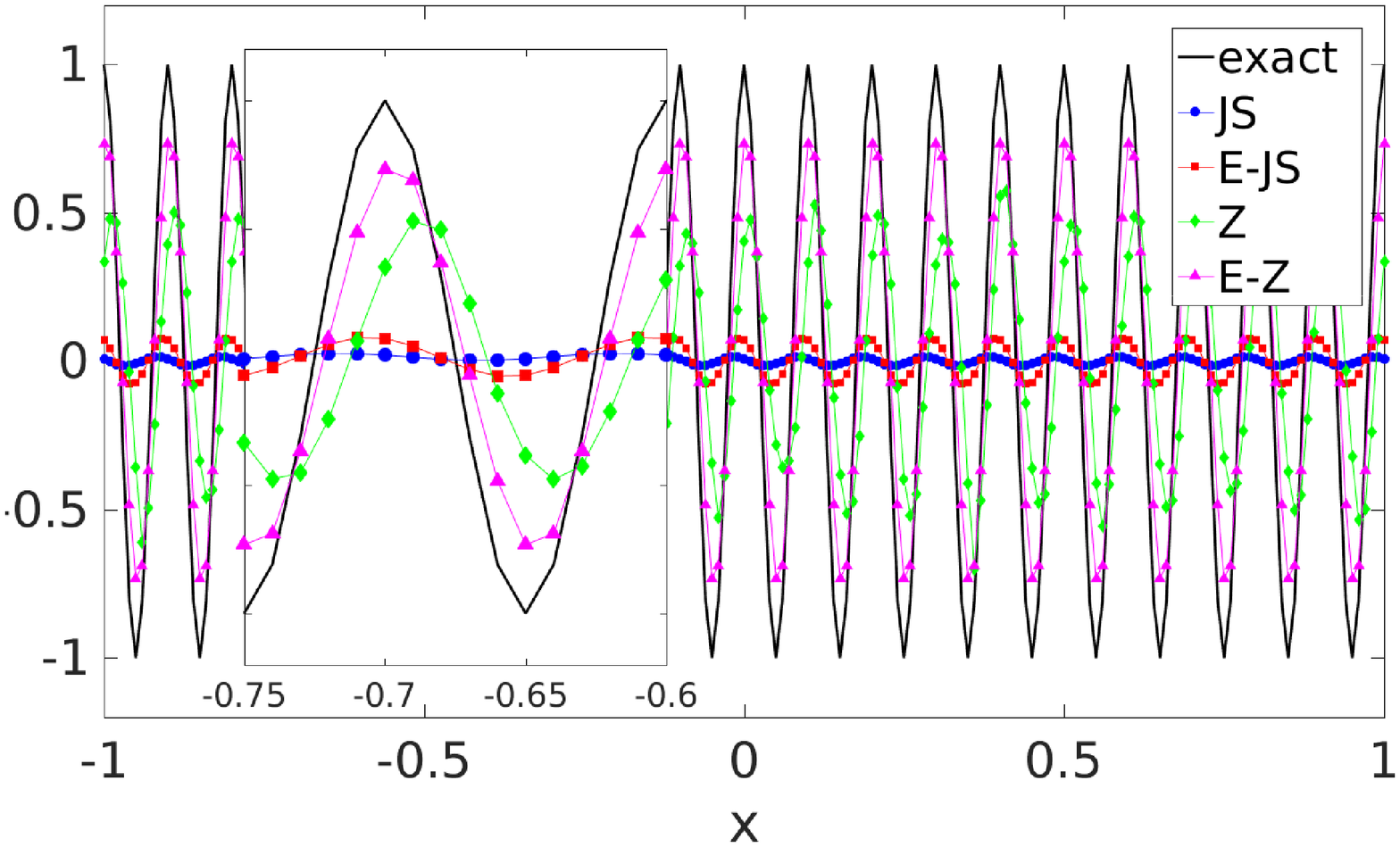}
\caption{Real part of the complex exponential plane wave with wave number $\kappa = 20 \pi$ at $t=4$ on a periodic domain with 201 grid points and a CFL number of $0.5$, resulting in 800 time steps.}
\label{fig:sine_wave_period20}
\end{figure}

\section{Numerical experiments}\label{sec:numerical_experiments}
As a final demonstration of the embedded WENO methods, we will perform some numerical experiments. 
An example of a scalar hyperbolic equation and an example of a hyperbolic system are presented. First, we shall take the linear advection equation with constant velocity field. Second, several cases of the Euler equations are numerically solved. In all examples we compare the embedded methods to their original counterparts. In all examples, we have used $\varepsilon = 10^{-12}$.

\subsection{Linear advection equation}
We consider the linear advection equation, i.e.,
\begin{equation}
\frac{\partial u}{\partial t} + \frac{\partial u}{\partial x} = 0,
\end{equation}
on $x\in [-1,1]$ with final time $t=2$. We use periodic boundary conditions, such that the initial condition is transported for one period and ends up where it started. Thus, the final state is equal to the initial condition, i.e., $u(x,2) = u_0(x)$. As an initial condition, we take the fairly standard test which uses a Gaussian, a square, a triangle and half an ellipse. This setup is sometimes referred to as the Shu linear test, introduced in \cite{jiang_shu}. The initial condition is given by
\begin{subequations}\label{eq:shu_initial}
\begin{equation}
u_0(x) = \begin{cases}
\tfrac{1}{6} \left( G(x;\beta,z-\epsilon) + G(x;\beta,z+\epsilon) + G(x;\beta,z) \right) & -0.8 \leq x \leq -0.6,\\
1 & -0.4 \leq x \leq -0.2,\\
1- |10(x-0.1)| & 0\leq x\leq 0.2,\\
\tfrac{1}{6} \left( F(x;\alpha,a-\epsilon) + F(x;\alpha,a+\epsilon) + 4 F(x;\alpha,a) \right) & 0.4 \leq x \leq 0.6,\\
0 & \text{otherwise},
\end{cases}
\end{equation}
where $G$ and $F$ are given by
\begin{align}
G(x;\beta,z) &:= \exp \big( - \beta(x-z)^2 \big),\\
F(x;\alpha,z) &:= \sqrt{\max \big( 1- \alpha^2(x-a)^2,0 \big) }.
\end{align}
\end{subequations}
The parameters are as follows: $z = -0.7$, $a = 0.5$, $\alpha = 10$, $\epsilon = \frac{1}{200}$, and $\beta
 = \frac{\ln 2}{36 \epsilon^2}$. One of the pervasive features of this test is the compact support of the initial condition. In fact, the shapes have non-overlapping supports. Thus, we need the numerical solutions to converge to zero as quickly as possible in between each shape. Thus, the third-order choice where the superfluous weight is shifted to the middle stencil offers the best choice heuristically. However, this will also provide more dissipation compared to the other options.\par
Let us start with the embedded schemes that use WENO-JS as its outer scheme, we expect to see better performance near discontinuities. Moreover, we also expect discontinuities in the first derivative to be captured better. The Shu linear test has both types of discontinuities, as well as smooth transitions. The embedded schemes switch to their inner schemes close to the edge of the support of each shape, and hence are able to capture it better, see Figure~\ref{fig:standard_advection_JS}.

\begin{figure}[h]
\centering
\includegraphics[width=\textwidth]{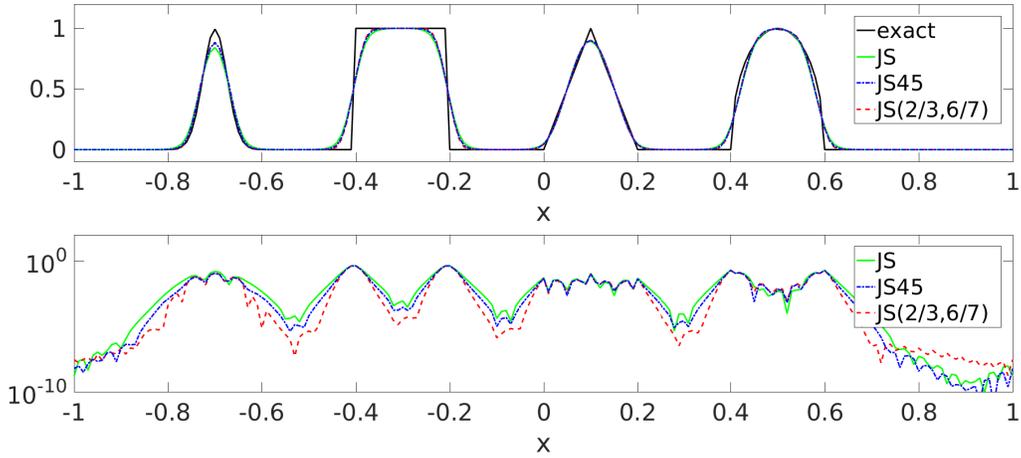}
\caption{Results for the linear advection equation (top) and errors (bottom). Shown are WENO-JS (green) and its embedded variants, WENO-45 (blue) and WENO-JS($\tfrac{2}{3}$,$\tfrac{6}{7}$) (red). Solutions were computed on a periodic domain at $t = 2$, using 201 grid points and a CFL number of $0.5$, resulting in 400 time steps.}
\label{fig:standard_advection_JS}
\end{figure}

It becomes clear from Figure~\ref{fig:standard_advection_JS} that the embedded WENO schemes perform better in almost every part of the domain. However, it should be noted that the WENO-45 scheme seems to perform best within each smoothly varying region, whereas the WENO-JS($\tfrac{2}{3}$,$\tfrac{6}{7}$) scheme captures the compact support of each shape the best. That is, the WENO-JS($\tfrac{2}{3}$,$\tfrac{6}{7}$) scheme seems to decay to zero the fastest in between each shape. However, the WENO-45 scheme has less dissipation and captures the maxima better in general. Interestingly, the WENO-JS scheme give the best representation of the peak of the triangle. However, in all other parts of the triangle, the embedded schemes provide a smaller error.\par
Next, we shall examine the performance of embedded schemes with the WENO-Z scheme as the outer scheme. By the same argument as presented previously, we expect the embedded schemes to perform better near discontinuities in the solution and its derivative. The results are plotted in Figure~\ref{fig:standard_advection_Z}. The figures show how the WENO-Z($\tfrac{2}{3}$,$\tfrac{6}{7}$) captures the compact support of the shapes the best. Again, this variant decays the fastest to zero in the space between the shapes.

\begin{figure}[h]
\centering
\includegraphics[width=\textwidth]{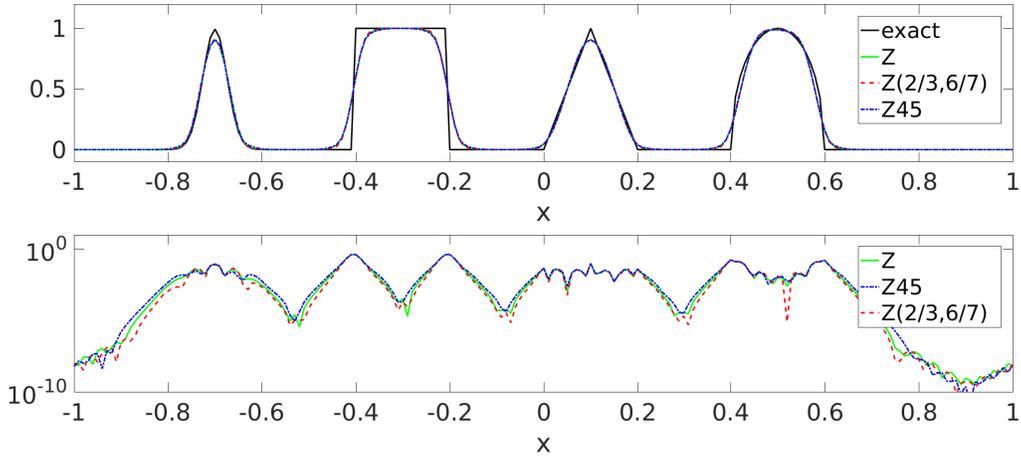}
\caption{Results for the linear advection equation (top) and errors (bottom). Shown are WENO-Z (green) and its embedded variants, WENO-Z45 (blue) and WENO-Z($\tfrac{2}{3}$,$\tfrac{6}{7}$) (red). Solutions were computed on a periodic domain at $t = 2$, using 200 grid points and a CFL number of $0.5$.}
\label{fig:standard_advection_Z}
\end{figure}

\subsection{Euler equations}\label{sec:euler}
Finally, we consider the one-dimensional Euler equations for ideal gases, i.e.,
\begin{equation}
\frac{\partial}{\partial t} \begin{pmatrix}
\rho \\
\rho u \\
E
\end{pmatrix} + \frac{\partial}{\partial x} \begin{pmatrix}
\rho u \\
\rho u^2 + p \\
u(E + p)
\end{pmatrix} = 0,
\end{equation}
with $\rho$ the density, $u$ the fluid velocity, $E$ the total energy and $p$ the pressure. We furthermore use the ideal caloric equation of state
\begin{equation}
E = \tfrac{1}{2} \rho u^2 + \frac{p}{\gamma-1} ,
\end{equation}
where $\gamma$ is the ratio of specific heats, which we fixed to $\gamma  = 1.4$ throughout. We employ the global Lax-Friedrichs flux splitting to construct the numerical flux and use the total variation diminishing Runge-Kutta time integrator from \eqref{eq:TVD-RK}. We shall first consider a collection of Riemann problems such as Sod's test, Lax's test and the 123-problem. We shall furthermore consider an interacting blast-wave problem of Woodward and Collela \cite{woodward} and the Mach-3 density-wave shock interaction problem of Shu and Osher \cite{shu_osher}.\par
In all examples of the Euler equations, at the beginning of every time step the maximum absolute eigenvalue of the Jacobian of the flux is computed and used to compute the time step size. This enforces that throughout the whole computation, the CFL number is kept at $0.4$. The WENO schemes were applied characteristic-wise to find the values on the cell edges. The Jacobian on the cell edge is computed using the arithmetic mean.\par 
We have compared the standard WENO-JS and WENO-Z schemes to their embedded variants WENO-JS($\tfrac{2}{3}$,$\tfrac{6}{7}$) and WENO-Z($\tfrac{2}{3}$,$\tfrac{6}{7}$) respectively, as these variants seemed to perform better for discontinuous solutions. As all the examples we cover contain discontinuities or sharp gradients of some kind, this is a natural choice.

\subsubsection{Riemann problems}
We consider here the Riemann problems of Sod's test, Lax's test and the 123-problem. A Riemann problem features a discontinuous initial condition with two states, i.e.,
\begin{equation}
(\rho,u,p) = \begin{cases}
(\rho_l,u_l,p_l) & \text{if } x<0,\\
(\rho_r,u_r,p_r) & \text{if } x>0.
\end{cases}
\end{equation}
This is the simplest possible non-trivial type of initial condition and for the Euler equations these types of problems can be solved exactly. Shock tube problems are a special type of Riemann problem with zero fluid velocity $u$ everywhere. All the Riemann problems are solved with $201$ grid points such that the initial discontinuity is exactly on a cell edge.\par
Sod's test is a shock tube problem with initial condition
\begin{equation}
(\rho,u,p) = \begin{cases}
(1,0,1) & \text{if } x<0,\\
(0.125,0,0.1) & \text{if } x>0,
\end{cases}
\end{equation}
with a final time of $t = 0.4$. We use as computational domain $x \in [-1,1]$, and so we use non-reflective boundaries at the edges of the domain. Sod's test is a very mild test, the exact solution consists of a left rarefaction wave, a contact discontinuity and a right shock.

\begin{figure}[h]
\centering
\includegraphics[width=\textwidth]{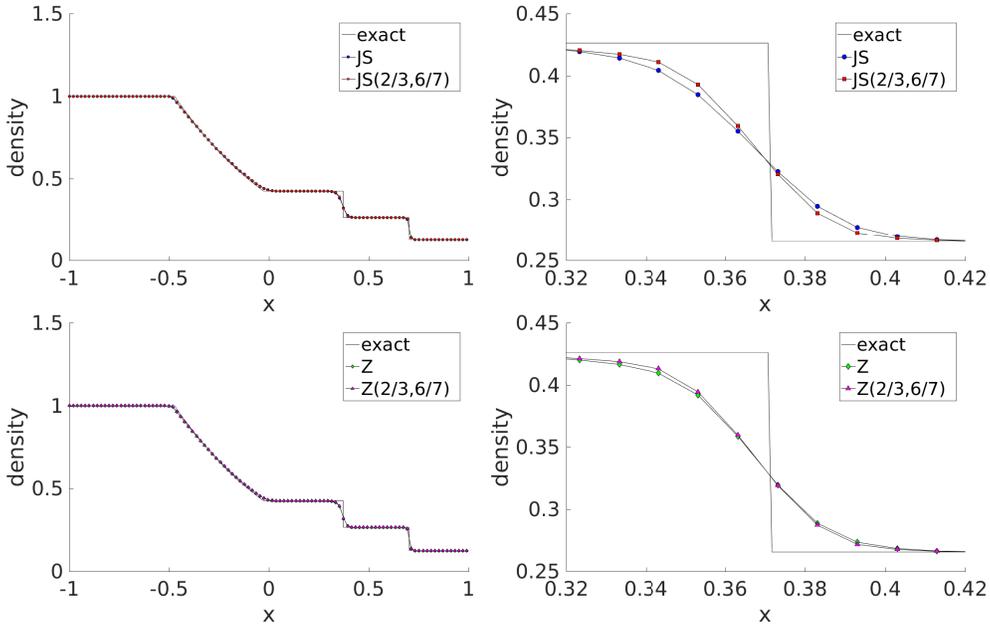}
\caption{The solution of Sod's test (left) and the result zoomed in on the contact discontinuity (right). The top figures show WENO-JS and WENO-JS($\frac{2}{3}$,$\frac{6}{7}$) while the bottom figures show WENO-Z and WENO-Z($\frac{2}{3}$,$\frac{6}{7}$). The numerical solutions were computed using 201 grid points and a CFL number of 0.4. The left-side plots are plotted with half the grid points for clarity. }
\label{fig:Sod}
\end{figure}

Figure~\ref{fig:Sod} shows that the embedded WENO schemes give a solution which is globally similar, as intended. However, zooming in on smaller features, the differences become clear. The embedded schemes give a slightly sharper gradient near the contact discontinuity as compared to their standard counterparts.\par
Lax's test has initial conditions
\begin{equation}
(\rho,u,p) = \begin{cases}
(0.445,0.689,3.528) & \text{if } x<0,\\
(0.5,0,0.5710) & \text{if } x>0,
\end{cases}
\end{equation}
with a final time of $t=0.25$. The exact solution again consists of a left rarefaction wave, a contact and a right shock. However, unlike Sod's test, the contact discontinuity has a rather large jump.

\begin{figure}[h]
\centering
\includegraphics[width=\textwidth]{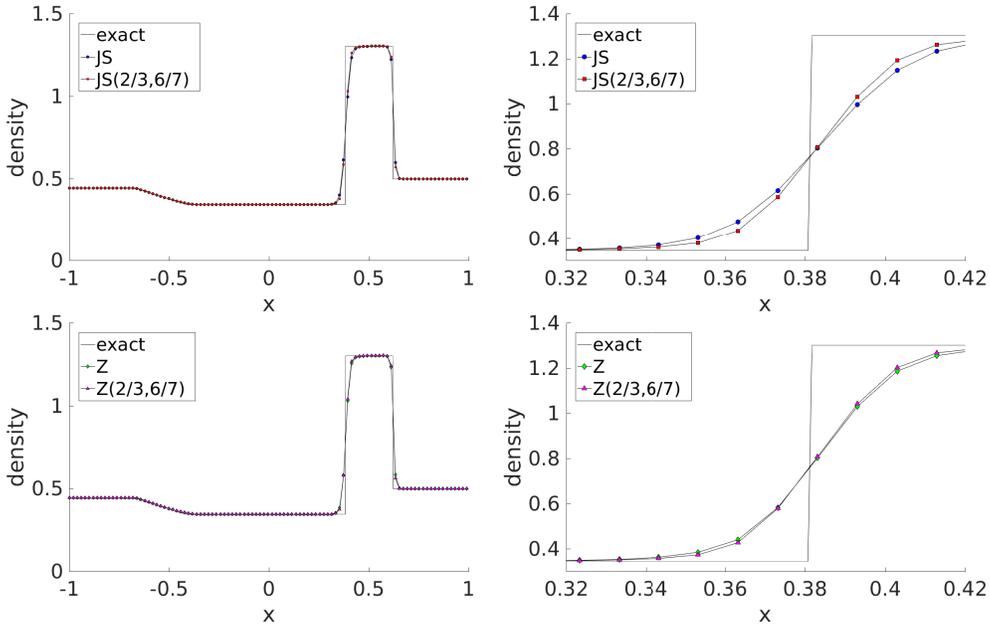}
\caption{The solution of Lax's test (left) and the result zoomed in on the contact discontinuity (right). The top figures show WENO-JS and WENO-JS($\frac{2}{3}$,$\frac{6}{7}$) while the bottom figures show WENO-Z and WENO-Z($\frac{2}{3}$,$\frac{6}{7}$). The numerical solutions were computed using 201 grid points and a CFL number of 0.4. The left-side plots are plotted with half the grid points for clarity. }
\label{fig:Lax}
\end{figure}

Also here, Figure~\ref{fig:Lax} shows how globally the embedded schemes give a similar solution. As Lax's test features a rather large jump in the contact discontinuity, we will zoom in on that part of the solution. Again, we see that both embedded schemes give a better representation of the discontinuity.\par
The 123-problem has initial conditions
\begin{equation}
(\rho,u,p) = \begin{cases}
(1.0,-2.0,0.4) & \text{if } x<0,\\
(1.0,2.0,0.4) & \text{if } x>0,
\end{cases}
\end{equation}
with a final time of $t=0.25$. The solution consists of two strong rarefactions with a trivial stationary contact, the pressure in between the two rarefaction waves is very low. As the exact solution consists of two rarefaction waves, the solution is rather smooth and hence the schemes perform roughly the same. The embedded schemes capture the transition between rarefaction wave and left or right state slightly better, see Figure~\ref{fig:123}.

\begin{figure}[h]
\centering
\includegraphics[width=\textwidth]{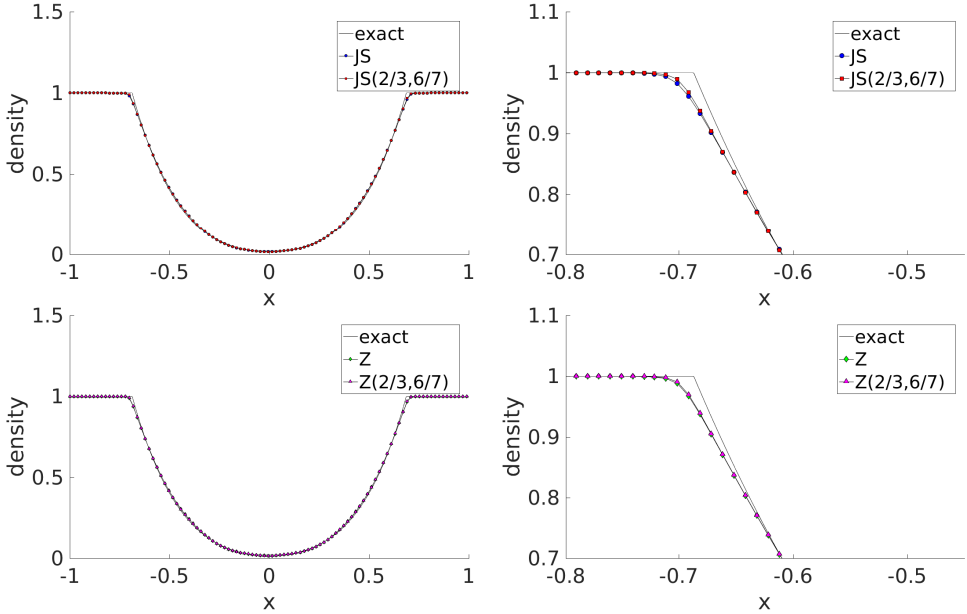}
\caption{The solution of the 123-problem (left) and the result zoomed in on the rarefaction wave (right). The top figures show WENO-JS and WENO-JS($\frac{2}{3}$,$\frac{6}{7}$) while the bottom figures show WENO-Z and WENO-Z($\frac{2}{3}$,$\frac{6}{7}$). The numerical solutions were computed using 201 grid points and a CFL number of 0.4. The left-side plots are plotted with half the grid points for clarity.}
\label{fig:123}
\end{figure}

\subsubsection{Interacting blast-waves}
Here we consider a problem featuring two interacting blast-waves proposed by Woodward and Colella \cite{woodward}. The computational domain is now $x \in [0,1]$ with reflective boundaries. The initial conditions have unit density and zero velocity in the entire domain. The pressure is set at 0.01 except for two small regions, where a very high pressure is present, i.e.,
\begin{equation}
(\rho,u,p) = \begin{cases}
(1.0,0.0,1000) & \text{if } 0 \leq x \leq 0.1,\\
(1.0,0.0,100) & \text{if } 0.9 \leq x \leq 1,\\
(1.0,0.0,0.01) & \text{otherwise}.
\end{cases}
\end{equation}
The final time is set to $t=0.038$. Both high pressure regions create blast-waves travelling outwards, which are reflected at the boundaries and immediately directed inwards. Complicated shapes in the density form before the blast-waves meet and interact. As there is no exact solution to this particular problem, we employed Godunov's method with $2 \cdot 10^4$ grid points and $10^5$ time steps to compute the reference solution.

\begin{figure}[h]
\centering
\includegraphics[width=\textwidth]{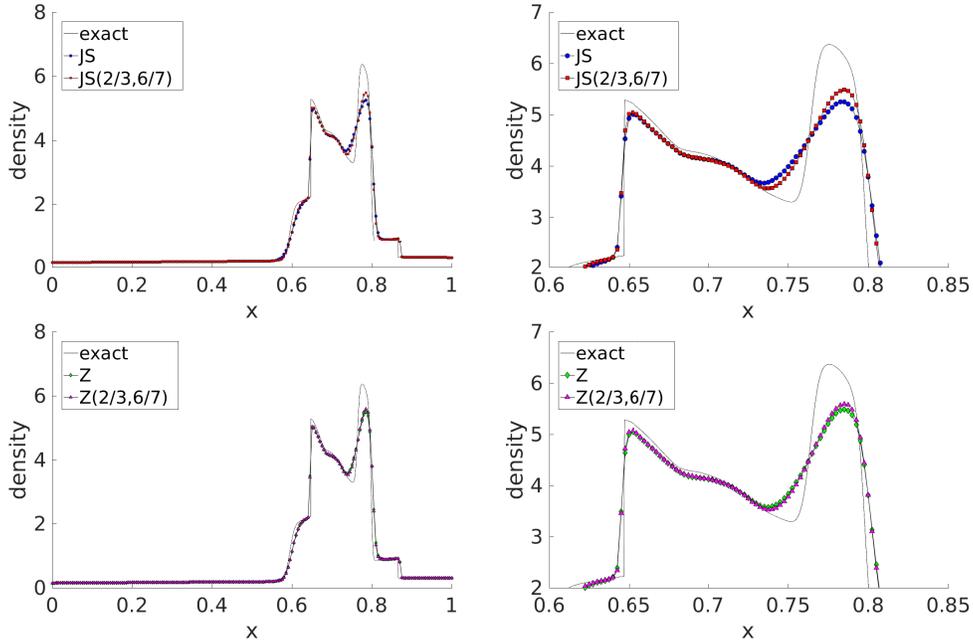}
\caption{The solution of the interacting blast-waves problem (left) and the result zoomed in on the interaction region (right). The top figures show WENO-JS and WENO-JS($\frac{2}{3}$,$\frac{6}{7}$) while the bottom figures show WENO-Z and WENO-Z($\frac{2}{3}$,$\frac{6}{7}$). The numerical solutions were computed using 400 grid points and a CFL number of 0.4. The left-side plots are plotted with half the grid points for clarity. }
\label{fig:WC}
\end{figure}

In this example, it is natural to focus attention on the region where the two blast-waves interact, see Figure~\ref{fig:WC}. In both cases, we see that the embedded schemes have higher peaks and lower valleys, and are thus closer to the reference solution.

\subsubsection{Mach-3 shock density-wave interaction}
The final problem under consideration is the Mach-3 shock density-wave interaction proposed by Shu and Osher \cite{shu_osher}. The initial conditions are given by
\begin{equation}
(\rho,u,p) = \begin{cases}
(3.857,2.629,10.333) & \text{if } x<0,\\
(1+\epsilon \sin(5x) ,0,1.0) & \text{if } x>0.
\end{cases}
\end{equation}
The integration time is $t=1.8$. If $\epsilon$ is set to zero, this is a Riemann problem with the solution being a pure Mach-3 shock wave travelling to the right. However, $\epsilon$ is set to 0.2, resulting in the right state being a regular density wave. Again, no exact solution is available, hence we use a numerical solution computed on a very fine grid, in this case the WENO-JS scheme with 2000 grid points.

\begin{figure}[h]
\centering
\includegraphics[width=\textwidth]{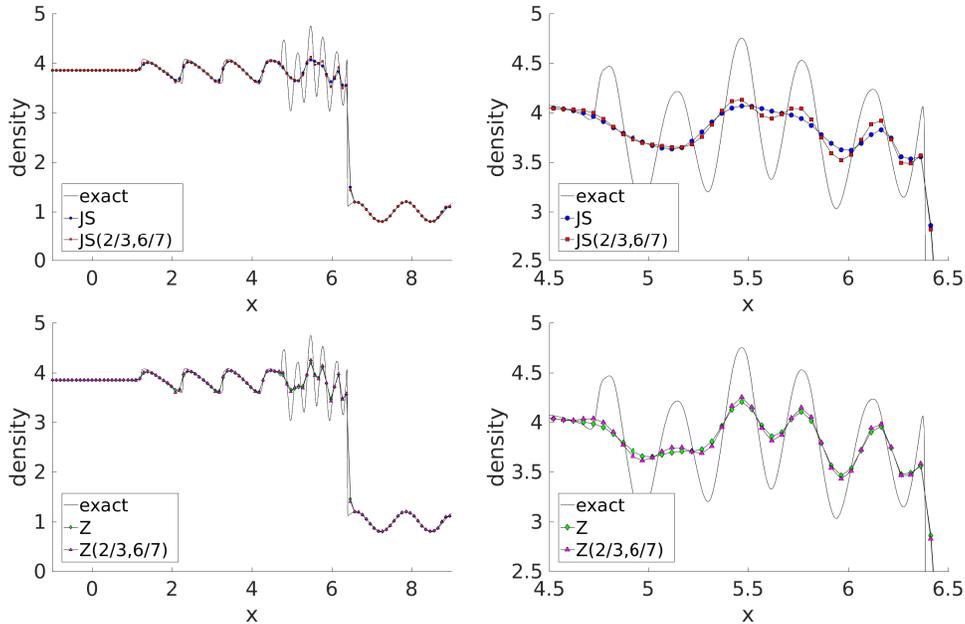}
\caption{The solution of the Mach-3 shock density-wave interaction test (left) and the result zoomed in on the interaction region (right). The top figures show WENO-JS and WENO-JS($\frac{2}{3}$,$\frac{6}{7}$) while the bottom figures show WENO-Z and WENO-Z($\frac{2}{3}$,$\frac{6}{7}$). The numerical solutions were computed using 201 grid points and a CFL number of 0.4. The left-side plots are plotted with half the grid points for clarity.}
\label{fig:M3DWSI_JS}
\end{figure}

In this final example, presented in Figure~\ref{fig:M3DWSI_JS}, we again see how the embedded schemes provide an improvement over their more standard-issue counterparts. With the WENO-JS scheme, the high-frequency density waves close to the shock are hardly captured, whereas the embedded scheme shows some more detail. Furthermore, the embedded scheme does a better job of capturing the low-frequency waves more to the left of the shock.\par
Similar things may be said about WENO-Z and its embedded version. Naturally WENO-Z shows more detail than WENO-JS, while the WENO-Z($\frac{2}{3}$,$\frac{6}{7}$) shows even more detail in the high-frequency region and less flattening in the low-frequency region.

\section{Conclusion and future work}\label{sec:conclusion}
We have introduced a design strategy for improving existing WENO weights and with it a new type of WENO methods, which we have named the embedded WENO methods. We have outlined a general approach that allows one to adapt the nonlinear weights of an existing WENO method. The overall, original, WENO method was called the outer method, while the adjustment was dubbed the inner method. The inner method takes over when several adjacent substencils are smooth while there is a discontinuity present in the larger stencil. In such regions, conventional WENO schemes, such as WENO-JS and WENO-Z, revert to their lower-order mode. This is slightly overzealous, as one has more than one smooth substencil to work with. The embedded WENO schemes switch to their inner scheme in these cases. This allows more control over the numerical solution, for instance by attaining a higher order of convergence.\par
A framework was presented along with four conditions that we dubbed the implementation, nonlocality, consistency and embedding conditions. The implementation and nonlocality conditions led us to the general forms. The consistency and embedding conditions provide equations for the coefficients of the correction when dealing with a particular WENO scheme. In this manner, we have explicitly constructed embedded schemes based on the five-point WENO-JS and the WENO-Z method.\par
We have demonstrated through spectral analysis and several numerical experiments the benefits of the embedded WENO schemes over their corresponding standard methods. All numerical examples show the same properties: equal or better performance in smooth regions and better performance near discontinuities in the solution and its derivative. We have also demonstrated that the spectral properties of a WENO scheme can be improved by converting it to an embedded version.\par
Whereas here we have applied our embedding strategy to the WENO-JS and WENO-Z schemes, we expect similar results when it is applied to other schemes. Our framework was presented in the context of five-point WENO schemes, but the consistency equations are easily generalised and the embedding equations were derived in a more general setting. In fact, we have constructed and demonstrated seven-point embedded WENO schemes elsewhere \cite{vanlith}. Therefore, we foresee no significant obstructions when applying the embedding strategy to other schemes.\par

\section*{Acknowledgements}
This work was generously supported by Philips Lighting and the Intelligent Lighting Institute. The authors would like to thank the referees for their helpful comments. The authors would like to thank especially Wai-Sun Don and Rafael Borges for the helpful discussions and usage of their Matlab code.


\end{document}